\documentclass[preprint]{elsarticle}
\usepackage[utf8]{inputenc}
\usepackage[]{natbib}
\setcitestyle{round, aysep={},authoryear}
\usepackage[english]{babel}
\usepackage{amsmath}
\usepackage{amsthm}
\usepackage{thmtools}
\usepackage{thm-restate}
\usepackage{bm}
\usepackage{mathtools}
\usepackage{xspace}
\usepackage{booktabs}
\usepackage{multirow}
\usepackage{amssymb}
\usepackage[hidelinks]{hyperref}
\usepackage{todonotes}
\usepackage{marvosym}
\usepackage[short, no comma]{optidef}
\usepackage{caption}
\usepackage{float}
\usepackage{mathrsfs}
\usepackage{subfig}
\usepackage{graphicx,adjustbox}
\usepackage[binary-units=true]{siunitx}
\pgfdeclareverticalshading{rainbow}{100bp}
{color(0bp)=(red); color(25bp)=(red); color(35bp)=(yellow);
	color(45bp)=(green); color(55bp)=(cyan); color(65bp)=(blue);
	color(75bp)=(violet); color(100bp)=(violet)}

\usepackage[a4paper,
totalwidth=.643\paperwidth,
totalheight=.93\paperheight,
headsep=.03\paperheight,
headheight = 13.59999pt,
hmarginratio=6:8,
vmarginratio=1:1,
includeheadfoot, 
bindingoffset=11mm,
dvips]{geometry}

\theoremstyle{definition}
\newtheorem{definition}{Definition}

\newtheorem{example}{Example}
\theoremstyle{plain}
\newtheorem{theorem}{Theorem}
\newtheorem{corollary}{Corollary}
\newtheorem{lemma}{Lemma}

\newcommand{\ProblemName}{\textsc{RobMCF$\equiv$}}
\newcommand{\RestrictedProblemName}[1]{\textsc{rRobMCF$\equiv\hspace*{-0.1cm}(\boldsymbol{\balanceDynPro_{#1}}, \boldsymbol{\costDynPro_{#1}})$}}
\newcommand{\Demand}[4]{\delta(#2^{#3}_{#4})}

\DeclareMathOperator{\freehelp}{free}
\DeclareMathOperator{\fixhelp}{fix}
\newcommand{\free}{A^{\freehelp}}
\newcommand{\fix}{A^{\fixhelp}}
\newcommand{\balanceDynPro}{\tilde{s}}
\newcommand{\costDynPro}{\tilde{c}}
\newcommand{\sourceSPGraph}{o}
\newcommand{\sinkSPGraph}{q}

\definecolor{rwth-blue}{RGB}{0,84,159}

\newenvironment{lyxlist}[1]
{\begin{list}{}
		{\settowidth{\labelwidth}{#1}
			\setlength{\leftmargin}{\labelwidth}
			\addtolength{\leftmargin}{\labelsep}
			}}
	{\end{list}}
\usepackage[Algorithm]{algorithm}
\usepackage[noend]{algpseudocode}

\makeatletter
\def\BState{\State\hskip-\ALG@thistlm}

\makeatother

\usepackage{lineno}
\modulolinenumbers[5]

\makeatletter
\let\@afterindenttrue\@afterindentfalse
\def\ps@pprintTitle{%
	\let\@oddhead\@empty
	\let\@evenhead\@empty
	\def\@oddfoot{}%
	\let\@evenfoot\@oddfoot}
\makeatother

\begin{document}
\begin{frontmatter}

\title{Robust Minimum Cost Flow Problem Under Consistent Flow Constraints}

\author{Christina B\"using}
\author{Arie M.C.A. Koster}
\author{Sabrina Schmitz\corref{mycorrespondingauthor}}

\ead{schmitz@math2.rwth-aachen.de}
\cortext[mycorrespondingauthor]{Corresponding author}
\address{Lehrstuhl II f\"ur Mathematik, RWTH Aachen University, Germany}

\begin{abstract}
The robust minimum cost flow problem under consistent flow constraints (\ProblemName{}) is a new extension of the minimum cost flow (MCF) problem.
In the \ProblemName{} problem, we consider demand and supply that are subject to uncertainty.
For all demand realizations, however, we require that the flow value on an arc needs to be equal if it is included in the predetermined arc set given.
The objective is to find feasible flows that satisfy the equal flow requirements while minimizing the maximum occurring cost among all demand realizations.

In the case of a discrete set of scenarios, we derive structural results which point out the differences with the polynomial time solvable MCF problem on networks with integral capacities.  
In particular, the Integral Flow Theorem of Dantzig and Fulkerson does not hold.
For this reason, we require integral flows in the entire paper. 
We show that the \ProblemName{} problem is strongly $\mathcal{NP}$-hard on acyclic digraphs by a reduction from the $(3,B2)$-\textsc{Sat} problem.
Further, we demonstrate that the \ProblemName{} problem is weakly $\mathcal{NP}$-hard on series-parallel digraphs by providing a reduction from \textsc{Partition} and a pseudo-polynomial algorithm based on dynamic programming.
Finally, we propose a special case on series-parallel digraphs for which we can solve the \ProblemName{} problem in polynomial time.
\end{abstract}

\end{frontmatter}

\section{Introduction}
\label{Sec:Introduction}
In this paper, we present a new extension of the minimum cost flow (MCF) problem~\citep{ahuja1988network}, which we call the \textit{robust minimum cost flow problem under consistent flow constraints} (\ProblemName{}). 
This problem is motivated by, for example, long-term decisions in logistic applications.
A major problem in logistics is the cost-efficient transport of commodities. 
Typically, this problem is represented by an MCF model, where a commodity can be identified by a flow sent through a network from a supply source to a sink with demand. 
In this way, a company can easily assess whether the available means of transport are sufficient for a given demand.
If this is not the case, additional transport by subcontractors can be arranged.
Such arrangements are generally agreed by long-term contracts, however, the demand is naturally subject to uncertainty. 
For this reason, valid and cost-efficient decisions have to be made without the knowledge of the actual demand.

The problem described can be represented by an adjusted integral MCF model subject to demand uncertainty.
We represent the demand uncertainty by a discrete number of possible occurring demand scenarios. 
In addition to the network requirements of the MCF problem, we are given a predetermined set of arcs referred to as fixed arcs. 
The flow value on the fixed arcs is supposed to represent the transport by subcontractors. 
Thus, we require the flow value on a fixed arc to be equal among all demand scenarios.
For finding a robust solution to this problem, we minimize the maximum cost that may occur among all demand scenarios.  
In summary, for the \ProblemName{} problem we consider different demand scenarios for which we require integral flows whose flow values are equal on the respective fixed arcs with the objective of minimizing the maximum cost.

The main contribution of this paper can be summarized as follows. 
We show that most of the knowledge of the MCF problem is not readily transferable to the \ProblemName{} problem. 
In particular, the integrality requirement of the \ProblemName{} problem is necessary, even though the network's capacities are integral, as Dantzig and Fulkerson's Integral Flow Theorem~\citep{korte2012combinatorial} does not hold. 
We further prove that the decision version of the \ProblemName{} problem is strongly $\mathcal{NP}$-complete on acyclic digraphs even if only two demand scenarios are considered. 
On series-parallel digraphs, we show that the decision version of the \ProblemName{} problem is weakly $\mathcal{NP}$-complete and solvable in pseudo-polynomial time by dynamic programming.
If in addition all demand scenarios have the same single source and sink, we propose an algorithm running in polynomial time. 
\\

The outline of this paper is as follows. 
We start with an overview about related work in Section~\ref{Sec:RelatedWork}. 
Subsequently, in Section~\ref{Sec:Problemdef}, we give an explicit mathematical problem description, and introduce the notations of this paper. 
Furthermore, we present first structural results of the problem. 
In Section~\ref{Subsec:ComplexityAcyclic}, the problem's complexity is analyzed on acyclic digraphs. 
Afterwards, we consider the \ProblemName{} problem on series-parallel digraphs in Section~\ref{Sec:SPDef}.
We conclude this paper by Section~\ref{Sec:Conclusion}.

\section{Related Work}
There are several related extensions to the MCF problem considered in the literature. 
In the following, we focus on extensions that consider equal flow requirements. 
Afterwards, we give a short overview of robust network approaches with demand uncertainty.
To the best of our knowledge, no study combines equal flow requirements with robust network approaches.

Sahni~(\citeyear{sahni1974computationally}) introduces a variant of the maximum flow problem~\citep{ahuja1988network}, the so-called \textit{integral flow with homologous arcs problem} (\textsc{homIF}). 
In addition to the set-up of the maximum flow problem, predetermined sets of arcs are given in this problem. 
A maximum integral flow is sought whose flow value is equal on all arcs that are contained in the same predetermined arc set. 
Sahni proves the $\mathcal{NP}$-hardness of the problem by a reduction from the \textsc{Non-Tautology} problem.
Garey and Johnson~(\citeyear{johnson1979computers}) point out that by modifying a construction of Even et al.~(\citeyear{even1975complexity}), the problem's $\mathcal{NP}$-hardness holds even if all arc capacities are equal to one. 
Furthermore, unless $\mathcal{P}=\mathcal{NP}$, the non-existence of a $2^{n(1-\varepsilon)}$-approximation algorithm for any fixed $\varepsilon>0$ (on a digraph with $n$ vertices) is proven by Meyers and Schulz~(\citeyear{meyers2009integer}) even if a nonzero solution is guaranteed to exist.

The MCF version of the \textsc{homIF} problem can be found in the literature as (integer) \textit{equal flow problem} (\textsc{EF}).
Using standard techniques, the complexity results can be transformed from the maximum flow to the MCF version~\citep{ahuja1988network}. 
There are several special cases and applications for both the maximum flow and MCF version of the problem considered in the literature~\citep{calvete2003network,meyers2009integer,morrison2013network,10.1007/3-540-45655-4_55}.  
For instance, the special case of the \textsc{EF} problem where all sets have cardinality two, i.e., an integral MCF is sought whose flow value is equal on a predetermined set of arc pairs, is investigated by Ali et al.~(\citeyear{ali1988equal}). 
The problem finds application in, for example, crew scheduling~\citep{carraresi1984network}.
Therefore, Ali et al.~present a heuristic algorithm based on Lagrangian relaxation.  
Meyers and Schulz~(\citeyear{meyers2009integer}) refer to this problem as \textit{paired integer equal flow problem} (\textsc{pEF}) and prove that there exists no $2^{n(1-\varepsilon)}$-approximation algorithm for any fixed $\varepsilon>0$ (on a digraph with $n$ vertices), unless $\mathcal{P}=\mathcal{NP}$. 
The statement holds true even if a nontrivial solution is guaranteed to exist.
A simpler and in polynomial time solvable special case of the \textsc{EF} problem is the so-called \textit{simple equal flow problem} (\textsc{sEF}), which is introduced by Ahuja et al.~(\citeyear{ahuja1999algorithms}). 
The \textsc{sEF} problem requires the same but not necessarily integral flow value on only a single predetermined set of arcs. 
The problem is motivated by the management of water resource systems~\citep{manca2010water}. 
For this purpose, Ahuja et al.~(\citeyear{ahuja1999algorithms}) develop several efficient algorithms to solve large-scale instances~--~a network simplex, a parametric simplex, a combinatorial parametric, a binary search, and a capacity
scaling algorithm.
These algorithms can easily be modified to obtain integral solutions.

Unlike the previous research on problems with equal flow requirements, we consider in the \ProblemName{} problem not one demand scenario only, but several demand scenarios.
For each of these scenarios, a feasible flow is sought. 
Furthermore, among all of these scenarios, we require the same flow value on an arc if it is included in the single predetermined set of fixed arcs.
Although we consider several demand scenarios, the problem of finding a feasible solution to the \ProblemName{} problem can be modeled as a special case of the \textsc{EF} and \textsc{pEF} problem by means of graph copies. 
We point out that the equal flow requirements in the \ProblemName{} problem are only of importance while considering different demand scenarios, i.e., the flow value of a fixed arc has to be equal among all scenarios.
In turn, the flow value of two fixed arcs may differ in one scenario.
For this reason, the problem of finding a feasible solution to the \ProblemName{} problem cannot be modeled as the \textsc{sEF} problem, except for the special case where the predetermined arc set contains only one arc. 
Moreover, due to different objectives, the correspondence from the \ProblemName{} problem to the \textsc{EF}, \textsc{sEF}, and \textsc{pEF} problem only holds for finding a feasible solution. 
\\

Demand uncertainty is studied more frequently in the context of network design and network engineering in telecommunication or road networks for example. 
In robust network design, we have to decide on the capacities such that in all considered scenarios, the entire demand can simultaneously be routed.
The cost of installing the capacities is supposed to be minimized. 

In the single commodity case which was first studied by Minoux~(\citeyear{Minoux1989}) and by Sanit\`{a}~\linebreak[3](\citeyear{Sanita2009}), the flow between supply and demand vertices may differ among the scenarios as long as the capacities are satisfied. For discrete scenarios, a cut-based integer linear program formulation with a separation algorithm is proposed by \'Alvarez-Miranda et al.~(\citeyear{AlvarezEtAl2012}). Cacchiani et al.~(\citeyear{cacchiani2016single}) present a branch-and-cut algorithm for two types of uncertainty sets, a discrete set of scenarios, and a polytope. Atamt\"urk and Zhang~(\citeyear{atamturk2007two}) present a two-stage robust optimization approach where some capacity decisions have to be made before, and other after the demand realization. The decisions have to guarantee that in any case the demand can be routed through the network.

In the multi-commodity case, several studies propose different models and uncertainty sets. 
For instance, Altin et al.~(\citeyear{AlAmBePi2007,AltinYamanPinar2011}) propose the so-called Hose uncertainty model, Belotti et al.~(\citeyear{BelottiCaponeCarelloMalucelli2008}) in the context of statistical multiplexing, and Koster et al.~(\citeyear{KoKuRa13}) the budget uncertainty set introduced by Bertsimas and Sim~(\citeyear{BertsimasSim2003,BertsimasSim2004}). In these studies, the flow is sent proportionally with the demand. In case the flow can be adapted to the demand, a two-stage robust approach is followed. While Mattia~(\citeyear{Mattia2013}) studies dynamic routing, Poss and Raack~(\citeyear{PossRaack2013}) suggest to use affine recourse options.

\label{Sec:RelatedWork}
\section{Robust Minimum Cost Flow Problem under Consistent Flow Constraints} 
\label{Sec:Problemdef}
\subsection{Definition \& Notation}	
\label{Subsec:Def}
The \ProblemName{} problem is an extension of the MCF problem where supply and demand is subject to uncertainty. 
The uncertainty is represented by a set of discrete scenarios $\Lambda$ where we do not have any knowledge which scenario is realized. 
Considering these scenarios, let a \textit{digraph} $G=(V,A)$ with vertex set $V$ and arc set $A$ be given. 
The set of arcs $A$ is defined by two disjoint sets, i.e., $A=A^\text{fix}\cup A^\text{free}$, where we refer to arcs of set $\fix$ as \textit{fixed arcs} and arcs of set $\free$ as \textit{free arcs}. 
Independent of the scenarios arc \textit{capacities} $u:A(G)\rightarrow \mathbb{Z}_{\geq 0}$ and arc \textit{cost} $c:A(G)\rightarrow \mathbb{Z}_{\geq 0}$ are given. 
In contrast, vertex \textit{balances} $b^\lambda:V(G)\rightarrow \mathbb{Z}$ with $\sum_{v\in V}b^\lambda(v)=0$ that define the supply and demand realizations are given for every scenario $\lambda\in \Lambda$, denoted by $\boldsymbol{b}=(b^1,\ldots,b^{|\Lambda|})$.
A positive balance indicates a \textit{source} while a negative balance indicates a \textit{sink}. 
Note that, in general, the source (sink) vertices do not necessarily have to be the same in every scenario.
In case that each scenario has only one vertex with a positive (negative) balance, we refer to these sources as \textit{single sources (sinks)}.
If the single sources (sinks) are defined by the same vertex for every scenario, we say that the problem has a \textit{unique source (sink)}. 
Combined, we obtain the \textit{network} $(G,u,c,\boldsymbol{b})$. 

Considering only a single scenario $\lambda\in \Lambda$, analogues to the MCF problem a \textit{$b^\lambda$-flow} is defined by a function $f^\lambda : A(G)\rightarrow  \mathbb{Z}_{\geq 0}$ that satisfies the \textit{capacity constraints}
$$0\le f^\lambda(a)\le u(a)$$ on all arcs $a\in A$
and the \textit{flow balance constraints }
\[
\sum_{a=(v,w)\in A}f^\lambda(a)-\sum_{a=(w,v)\in A}f^\lambda(a)=b^\lambda(v)
\]
at every vertex $v\in V$. 
The cost of a $b^\lambda$-flow $f^\lambda$ is defined by $$c(f^\lambda)=\sum_{a\in A} c(a)\cdot f^\lambda(a).$$
To consider a set of scenarios $\Lambda$, we need to introduce a new definition of a flow, a so called \textit{robust $\boldsymbol{b}$-flow}.
\begin{definition}[Robust Flow]
Given a network $(G=(V,A=A^{\text{fix}}\cup A^{\text{free}}),u,c, \boldsymbol{b})$, a \textit{robust $\boldsymbol{b}$-flow} $\boldsymbol{f}=(f^1,\ldots,f^{|\Lambda|})$ is defined by a $|\Lambda|$-tuple of $b^\lambda$-flows $f^\lambda: A(G)\rightarrow  \mathbb{Z}_{\geq 0} $ that satisfy the \textit{consistent flow constraints} $f^\lambda(a)=f^{\lambda^\prime}(a)$ on all fixed arcs $a\in A^\text{fix}$ for all scenarios $\lambda, \lambda^\prime \in \Lambda$. The cost of a robust $\boldsymbol{b}$-flow $\boldsymbol{f}$ is defined by the maximum flow cost among all scenarios, i.e., $c(\boldsymbol{f})=\max_{\lambda\in \Lambda}c(f^\lambda)$.
\end{definition}
We refer to the flow value on an arc of set $\fix$ as its \textit{load}. 
Accordingly, the consistent flow constraints are satisfied if the load of a fixed arc is equal in every scenario. 
The \ProblemName{} problem can finally be formulated as follows.
\begin{definition}[\ProblemName{}]
Given a network $(G=(V,A=A^{\text{fix}}\cup A^{\text{free}}),u,c,\boldsymbol{b})$, the \textit{robust minimum cost flow problem under consistent flow constraints} \linebreak[3] (\ProblemName{}) is to find a robust $\boldsymbol{b}$-flow $\boldsymbol{f}=(f^1,\ldots, f^{|\Lambda|})$ of minimum cost.
\end{definition}
Note that in case of a single scenario, i.e., $|\Lambda|=1$, the \ProblemName{} problem corresponds to the MCF problem.  
Otherwise, however, there are major differences as the following section shows.

\subsection{Structural Results}	
\label{Subsec:StrucRes}
In this section, we present structural results of the \ProblemName{} problem. 
In particular, differences to the MCF problem are pointed out where the main difference is the following.
Given a network with integral arc capacities, by Dantzig and Fulkerson~\citep{korte2012combinatorial} there always exists an optimal integral flow for the MCF problem.
This useful integral flow property is assumed to be given in most studies. However, the integral flow property does not hold for the \ProblemName{} problem as the following example shows. 
\begin{example}
For a set of two scenarios $\Lambda=\{1,2\}$, let a network $(G,u,c,\boldsymbol{b})$ with capacity $u\equiv1$ be given, where digraph $G$, its cost $c$, and the non-zero balances $\boldsymbol{b}$ are visualized in Figure~\ref{fig:IntegralityExample}.
An optimal integral robust $\boldsymbol{b}$-flow $\boldsymbol{f}=(f^1,f^2)$ is defined by a first scenario flow $f^1$ that sends one unit along path $v_1v_3v_5v_8$, and a second scenario flow $f^2$ that sends one unit each along path $v_1v_3v_5v_6$ and $v_1v_4v_5v_7v_6$. 
This results in cost of $c(\boldsymbol{f})= 4$ as $c(f^1)=4$ and $c(f^2)=2+0=2$ hold true.

However, by neglecting the integral flow requirement, there exists a robust $\boldsymbol{b}$-flow $\boldsymbol{\tilde{f}}=(\tilde{f}^1,\tilde{f}^2)$ with cost of $c(\boldsymbol{\tilde{f}})=3$.
Flow $\tilde{f}^1$ sends a half unit each along paths $v_1v_3v_5v_8$ and $v_1v_4v_5v_8$ ending up in total cost of $c(\tilde{f}^1)=3$.
Flow $\tilde{f}^2$ sends a half unit each along paths $v_1v_2v_5v_6$ and $v_1v_3v_5v_6$, and one unit along path $v_1v_4v_5v_7v_6$, also ending up in cost of $c(\tilde{f}^2)=3$. 
\end{example}
\begin{figure}
	\begin{adjustbox}{max width=1\textwidth, max height=1\textheight}
		\subfloat[]
		{\label{fig:IntegralityExample}\includegraphics{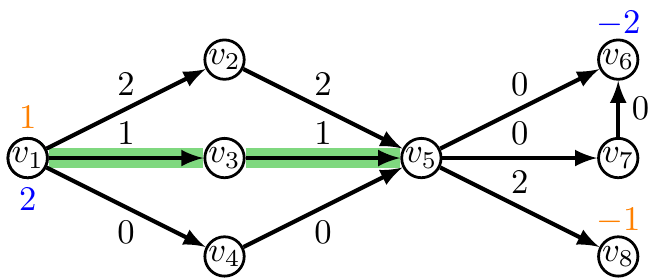}}
		\hspace*{1cm}
		\subfloat[]
		{\label{fig:ExampleCost}\includegraphics{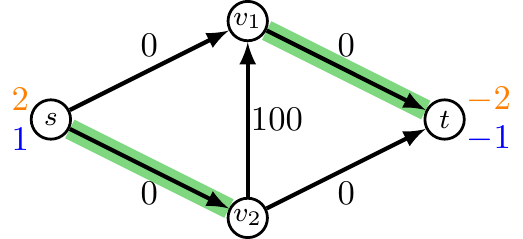}}
	\end{adjustbox}
\centering\includegraphics{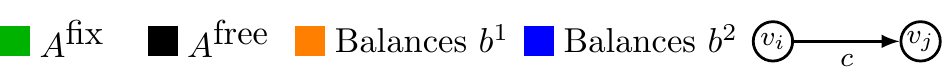}
\caption{(a) A non-integral robust $\boldsymbol{b}$-flow is the only optimal solution (b) A scenario flow that does not send the maximum demand among all scenarios causes the maximum cost in a unique source, unique sink network}
\end{figure}
\begin{corollary}
Considering the continuous relaxation of the \ProblemName{} problem, there does not always exist an integral robust flow with minimum cost even if all arc capacities are integral.
\end{corollary}
We note that if no integer requirements for a robust flow are given, the \ProblemName{} problem can be solved by a simple linear program (LP) in polynomial time in $|V|$,$|A|$ and $|\Lambda|$.
However, as applications of the \ProblemName{} problem often require integral flow values, hereafter this paper only concentrates on integral solutions. 
Further motivated by applications, in the next step, we investigate the \ProblemName{} problem where either the load of the fixed arcs is given, or the number of fixed arcs is constant. 
In logistics for example, this complies with finding a solution of minimum cost if the transport is already contractually agreed or limited. 
The following results show that we can solve these special cases in polynomial time. 

\begin{lemma}\label{lem:Transformation}
Let $\mathcal{I}=(G=(V,A=\fix \cup \free),u,c,\boldsymbol{b})$ be a \ProblemName{} instance. 
For a given load $\ell: \fix(G) \rightarrow \mathbb{Z}_{\geq 0}$, 
an optimal robust $\boldsymbol{b}$-flow $\boldsymbol{f}$ that satisfies $f^\lambda(a)=\ell(a)$ for all fixed arcs $a\in \fix(G)$ in every scenario $\lambda\in \Lambda$ can be computed in polynomial time if one exists.  
\end{lemma}
\begin{proof}
We transform instance $\mathcal{I}$ to $|\Lambda|$ simple minimum cost $\tilde{b}^\lambda$-flow instances $\widetilde{\mathcal{I}}^\lambda=(\widetilde{G},u,c,\linebreak[3]\tilde{b}^\lambda)$ that can be considered for every scenario $\lambda \in \Lambda$ separately. 
Instances $\widetilde{\mathcal{I}}^\lambda$, $\lambda\in \Lambda$ are obtained by deleting the fixed arcs from digraph $G$ resulting in digraph $\widetilde{G}$, i.e., $\widetilde{G}=G-\fix$, while at the same time the new balances $\tilde{b}^\lambda:V(\widetilde{G})\rightarrow \mathbb{Z}$ are defined as follows
\begin{align*}
		\tilde{b}^\lambda(v)= b^\lambda(v) + \sum_{a=(w,v)\in \fix(G)}\ell (a) - \sum_{a=(v,w)\in \fix(G)} \ell(a).
\end{align*}
After computing minimum cost $\tilde{b}^\lambda$-flows $\tilde{f}^\lambda$ for all instances $\widetilde{\mathcal{I}}^\lambda$, $\lambda\in \Lambda$, a corresponding robust $\boldsymbol{b}$-flow $\boldsymbol{f}=(f^1,\ldots,f^{|\Lambda|})$ for instance $\mathcal{I}$ is defined as 
\begin{align*}
f^\lambda(a)=
\begin{cases}
\ell(a) 				&\text{ for all fixed arcs }a\in \fix(G),\\
\tilde{f}^\lambda(a) 	&\text{ for all free arcs }a\in \free(G),
\end{cases}
\end{align*}
and causes cost of $$c(\boldsymbol{f})=\max_{\lambda\in \Lambda} c(\tilde{f}^\lambda) + \sum_{a\in \fix(G)} c(a)\cdot \ell(a).$$
Assume the constructed robust $\boldsymbol{b}$-flow $\boldsymbol{f}$ is not optimal, i.e., a robust $\boldsymbol{b}$-flow $\boldsymbol{\hat{f}}=(\hat{f}^1,\ldots,\hat{f}^{|\Lambda|})$ exists with cost
\begin{align*}
c(\hat{f}^{\lambda_1})=\max_{\lambda\in \Lambda}c(\hat{f}^\lambda)=c(\boldsymbol{\hat{f}})<c(\boldsymbol{f})=\max_{\lambda\in \Lambda}c(f^\lambda)=c(f^{\lambda_2})
\end{align*}
for scenarios $\lambda_1, {\lambda_2}\in \Lambda$.
Let $\overline{f}^\lambda$ 
denote the flow which results from restricting the scenario flow ${\hat{f}}^\lambda$ of instance $\mathcal{I}$ to instance $\widetilde{\mathcal{I}}^\lambda$, $\lambda\in \Lambda$. 
As the load on the fixed arcs is given, the values of flows $\boldsymbol{\hat{f}}$ and $\boldsymbol{f}$ on the fixed arcs are equal for every scenario, i.e., $\hat{f}^\lambda(a)=f^\lambda(a)=\ell(a)$ for $a\in \fix $ and $\lambda\in \Lambda$. 
Using this insight, we obtain
\begin{align*}
	&& c(\hat{f}^{\lambda_1})
	&<c(f^{\lambda_2}) &\\
	\Leftrightarrow && \sum_{a\in A(G)}c(a) \hat{f}^{\lambda_1}(a) 
	&<\sum_{a\in A(G)} c(a)  f^{\lambda_2}(a) &\\
	\Leftrightarrow  && \sum_{a\in \free(G)}c(a) \hat{f}^{\lambda_1}(a) 
	&<\sum_{a\in \free(G)} c(a)  f^{\lambda_2}(a)&\\
	\Leftrightarrow && c(\overline{f}^{\lambda_1}) 
	&<c(\tilde{f}^{\lambda_2}).
\end{align*}
Furthermore, as by definition flow $\boldsymbol{\hat{f}}$ satisfies the consistent flow constraints,
\linebreak[3] $c(\overline{f}^{\lambda_1})=\max_{\lambda\in \Lambda}c(\overline{f}^\lambda)$ is implied by $c(\hat{f}^{\lambda_1})=\max_{\lambda\in \Lambda}c(\hat{f}^\lambda)$. 
Overall, we obtain
\begin{align*}
c(\tilde{f}^{\lambda_2})
>c(\overline{f}^{\lambda_1})
=\max_{\lambda\in \Lambda}c(\overline{f}^\lambda)
\geq c(\overline{f}^{\lambda_2}),
\end{align*}
which is a contradiction to the fact that flow $\tilde{f}^{{\lambda_2}}$ is an optimal $\tilde{b}^{\lambda_2}$-flow for instance $\widetilde{\mathcal{I}}^{\lambda_2}$.

Considering the runtime, the transformation of instance $\mathcal{I}$ to instances $\widetilde{\mathcal{I}}$, $\lambda\in \Lambda$ is done in $\mathcal{O}(|\Lambda|\cdot|A|)$ time. 
Subsequently, a minimum cost $\tilde{b}^\lambda$-flow $\tilde{f}^\lambda$ can be computed for every scenario $\lambda\in \Lambda$ by, for example, the Minimum Mean Cycle-Canceling Algorithm in $\mathcal{O}(|A|^3|V|^2\log |V|)$ time~\citep{korte2012combinatorial}. 
Hence, an optimal robust $\boldsymbol{b}$-flow can be computed in $\mathcal{O}(|\Lambda|\cdot|A|^3\cdot|V|^2\log |V|)$ total time.
Note that if for a scenario $\lambda\in \Lambda$ no feasible $\tilde{b}^\lambda$-flow exists, there also does not exist a robust $\boldsymbol{b}$-flow. 
\end{proof}
\begin{corollary}
The \ProblemName{} problem is solvable in polynomial time for a constant number of fixed arcs. 
\end{corollary}
\begin{proof}
We formulate the \ProblemName{} problem as LP where we only require the constant number of variables that indicate the load on the fixed arcs to be integral. 
The resulting mixed integer linear program can be solved in polynomial time by Lenstra's algorithm~(\citeyear{lenstra1983integer}). 
In case that the resulting robust flow is not integral, we can find an integral flow with equal cost by Lemma~\ref{lem:Transformation} in polynomial time. 
\end{proof}

At the end of this section, we focus on the objective function of the \ProblemName{} problem. 
From the MCF problem, or the multi-commodity flow problem~\citep{korte2012combinatorial}, we know that due to different sources and sinks a flow that sends one unit may cause higher cost than a flow sending two units. 
Obviously, the same property remains true for instances of the \ProblemName{} problem.
If we consider the \ProblemName{} problem on networks with a unique source and a unique sink, we might assume, analogous to the MCF problem, that the cost of a robust flow is determined by the scenario flow which sends the maximum demand.
However, the following example shows that this is not true. 
\begin{example}\label{expl:ExampleCost}
For a set of two scenarios $\Lambda=\{1,2\}$, let a network $(G,u,c,\boldsymbol{b})$ with capacity $u\equiv1$ be given, where digraph $G$, its cost $c$, and the non-zero balances $\boldsymbol{b}$ are visualized in Figure~\ref{fig:ExampleCost}.
The only feasible and therefore also optimal solution $\boldsymbol{f}=(f^1, f^2)$ to the \ProblemName{} problem is easy to determine.
Considering the second scenario flow $f^2$ first, the only option to send two flow units from source $s$ to sink $t$ is along paths $sv_1t$ and $sv_2t$ due to the capacity constraints. 
As the second scenario flow $f^2$ uses both fixed arcs, the first scenario flow $f^1$ must also send flow along these arcs. 
For this reason, the only option to send one flow unit from source $s$ to sink $t$ is along path $sv_2v_1t$.
The cost of the robust $\boldsymbol{b}$-flow $\boldsymbol{f}$ is $c(\boldsymbol{f})=c(f^1)=100$. 
\end{example}
\begin{corollary}\label{cor:cost}
In a network with a unique source and a unique sink, the cost of a robust $\boldsymbol{b}$-flow is not necessarily determined by the $b^\lambda$-flow which sends the maximum demand among all scenarios $\lambda\in \Lambda$.
\end{corollary}
As a result, independent of the number of sources and sinks given, for solving the \ProblemName{} problem, we cannot only concentrate on a single scenario.
However, by reason of the following lemma, in a network with a unique source and a unique sink it is sufficient to concentrate on two scenarios only, namely those in which the minimum and maximum demand is sent.  
\begin{lemma}\label{lem:UniqueSourceUniqueSinkLimitationCost}
Let $\mathcal{I}=(G=(V,A=\free\cup \fix),u,c,\boldsymbol{b})$ be a \ProblemName{} instance with a unique source $s$ and a unique sink $t$. Without loss of generality, let the scenarios $\lambda\in \Lambda$ be strictly ordered in ascending order of their supply balances $b^\lambda$, i.e., $b^{1}(s)<b^{2}(s)<\ldots< b^{|\Lambda|}(s)$. 
Further, let feasible integral $b^\lambda$-flows $f^\lambda$ for scenarios $\lambda=1$ and $\lambda=|\Lambda|$ be given that satisfy the consistent flow constraints, i.e., $f^{1}(a) = f^{|\Lambda|}(a)$ for $a\in A^{\text{fix}}$.
A robust $\boldsymbol{b}$-flow $\boldsymbol{f}=(f^1,\ldots,f^{|\Lambda|})$ with cost of $c(\boldsymbol{f})=\max\{c(f^1), c(f^{|\Lambda|})\}$ can be computed in polynomial time.
\end{lemma}
\begin{proof}
As we consider a network with a unique source and a unique sink, a feasible robust $\boldsymbol{b}$-flow $\boldsymbol{f}$ for instance $\mathcal{I}$ is given by the convex combination of the flows $f^1$ and $f^{|\Lambda|}$ as follows.
For every scenario $\lambda\in \Lambda\setminus \{1,|\Lambda|\}$ let $\gamma^\lambda\in [0,1]$ be a parameter such that 
\begin{align*}
b^\lambda(s)= \gamma^\lambda \cdot b^1(s) + (1-\gamma^\lambda) \cdot b^{|\Lambda|}(s) 
\end{align*}
holds. 
We define the corresponding scenario flows $f^\lambda$, $\lambda\in \Lambda\setminus \{1,|\Lambda|\}$ by
\begin{align*}
f^\lambda(a):= \gamma^\lambda \cdot f^1(a) + (1-\gamma^\lambda) \cdot f^{|\Lambda|}(a) 
\end{align*}
for all arcs $a \in A$. 
Flows $f^\lambda$, $\lambda\in \Lambda\setminus \{1,|\Lambda|\}$ satisfy the capacity and flow balance constraints, but may be non-integral on some free arcs.
Therefore, we restrict every $b^\lambda$-flow $f^\lambda$, $\lambda\in \Lambda\setminus \{1,|\Lambda|\}$ to the respective MCF instance $\widetilde{\mathcal{I}}^\lambda$ obtained analogous to the proof of Lemma~\ref{lem:Transformation}, and this results in feasible $\tilde{b}^\lambda$-flows denoted by $\tilde{f}^\lambda$. 
Let $\tilde{f}^{\lambda}_{\text{OPT}}$ be an optimal integral $\tilde{b}^\lambda$-flow for instance $\widetilde{\mathcal{I}}^\lambda$, $\lambda\in \Lambda\setminus \{1,|\Lambda|\}$, then $c(\tilde{f}^\lambda_{\text{OPT}})\leq c(\tilde{f}^{\lambda})$ holds true. 
For all scenarios $\lambda\in \Lambda\setminus\{1,|\Lambda|\}$, flows $\tilde{f}^{\lambda}_{\text{OPT}}$ and $\tilde{f}^\lambda$ can be retransformed to flows $f^\lambda_{\text{OPT}}$, and respectively, $f^\lambda$ of instance $\mathcal{I}$, ending up in cost of 
\begin{align*}
c(f^\lambda_{\text{OPT}})	&= c(\tilde{f}^{\lambda}_{\text{OPT}} ) + \sum_{a\in A^{\text{fix}}} c(a) f^1(a)\\
					&\leq c(\tilde{f}^{\lambda}) + \sum_{a\in A^{\text{fix}}} c(a) f^1(a)\\
					&= c(f^{\lambda})\\
					&=\gamma^\lambda \cdot c(f^1) + (1-\gamma^\lambda) \cdot  c(f^{|\Lambda|})\\
					&\leq \max\{c(f^1), c(f^{|\Lambda|})\}.
\end{align*}
Consequently, an optimal robust $\boldsymbol{b}$-flow $\boldsymbol{f_{\textbf{OPT}}}=(f^1,f^2_{\text{OPT}},\ldots,f^{|\Lambda|-1}_{\text{OPT}}, f^{|\Lambda|})$ with cost $c({\boldsymbol{f}_{\textbf{OPT}}})=\max\{c(f^1), c(f^{|\Lambda|})\}$ is obtained in polynomial time analogous to the proof of Lemma~\ref{lem:Transformation}.
\end{proof}
As a result of Lemma \ref{lem:UniqueSourceUniqueSinkLimitationCost}, if a network with a unique source and a unique sink is given, we only need to concentrate on the first and last scenario to solve the \ProblemName{} problem. We obtain the following conclusion about the problem's complexity which is detailed in the next section. 
\begin{corollary}
For a set of scenarios $\Lambda$ with $|\Lambda|\geq 2$, let a \ProblemName{} instance with a unique source and unique sink be given. The complexity is not influenced by the number of scenarios. 
\end{corollary}
\section{Complexity for Acyclic Digraphs}	
\label{Subsec:ComplexityAcyclic}
In this section, we investigate the complexity of the \ProblemName{} problem for networks based on acyclic digraphs. 
For convenience, we discuss the problem's complexity for networks with a unique source and multiple sinks first. 
The construction is extended to show the strong $\mathcal{NP}$-completeness for networks with a unique source and a unique sink.
Both reductions are performed from the $(3,B2)$-\textsc{Sat} problem~\citep{berman2004approximation}~--~a strongly $\mathcal{NP}$-complete special case of the $3$-\textsc{Sat} problem.
We start with formal definitions for the decision versions of the \ProblemName{} and the $(3,B2)$-\textsc{Sat} problem. 
\begin{definition}
	The decision version of the \ProblemName{} problem asks whether a robust flow exists with cost at most $\beta\in \mathbb{Z}_{\geq 0}$. 
\end{definition}
\begin{definition}[$(3,B2)$-\textsc{Sat}]
Let $\{x_1,\ldots,x_n\}$ be a set of variables. Further, let $C_1,\ldots,C_m$ be a collection of clauses of size three where every positive and negative literal $x_i$ and $\overline{x}_i$ occur exactly twice. 
The decision problem of $(3,B2)$-\textsc{Sat} asks if there exists a variable assignment that satisfies the collection of clauses. 
\end{definition} 
Using the $(3,B2)$-\textsc{Sat} problem, we obtain the following complexity result.
\begin{theorem}\label{Theorem:ReductionAcyclicGraphs}
Deciding whether a feasible solution to the \ProblemName{} problem exists for networks based on acyclic digraphs with a unique source but multiple sinks is strongly $\mathcal{NP}$-complete even if only two scenarios are considered.
\end{theorem}
For the sake of clarity, we use the notation $[n]:=\{1,\ldots,n\}$ in the following.
\begin{proof}
The \ProblemName{} problem is contained in $\mathcal{NP}$ as we can check in polynomial time whether the flow balance, capacity, and consistent flow constraints are satisfied for every scenario. Further, we show that deciding whether a feasible solution of the \ProblemName{} problem exists is strongly $\mathcal{NP}$-complete by a reduction from the $(3,B2)$-\textsc{Sat} problem.  

Let $\mathcal{I}$ be a $(3,B2)$-\textsc{Sat} instance with the set of variables $\{x_1,\ldots,x_n\}$ and clauses $C_1,\ldots,C_m$ for which we construct a corresponding \ProblemName{} instance $\widetilde{\mathcal{I}}= (G, u, c, \boldsymbol{b})$ considering a set of two scenarios, i.e., $\Lambda=\{1,2\}$.
An example of a \ProblemName{} instance corresponding to a $(3,B2)$-\textsc{Sat} instance with four clauses and three variables is visualized in Figure~\ref{Reduction}.
\begin{figure}
	\begin{adjustbox}{max width=1\textwidth, max height=1\textheight}
		\centering\includegraphics{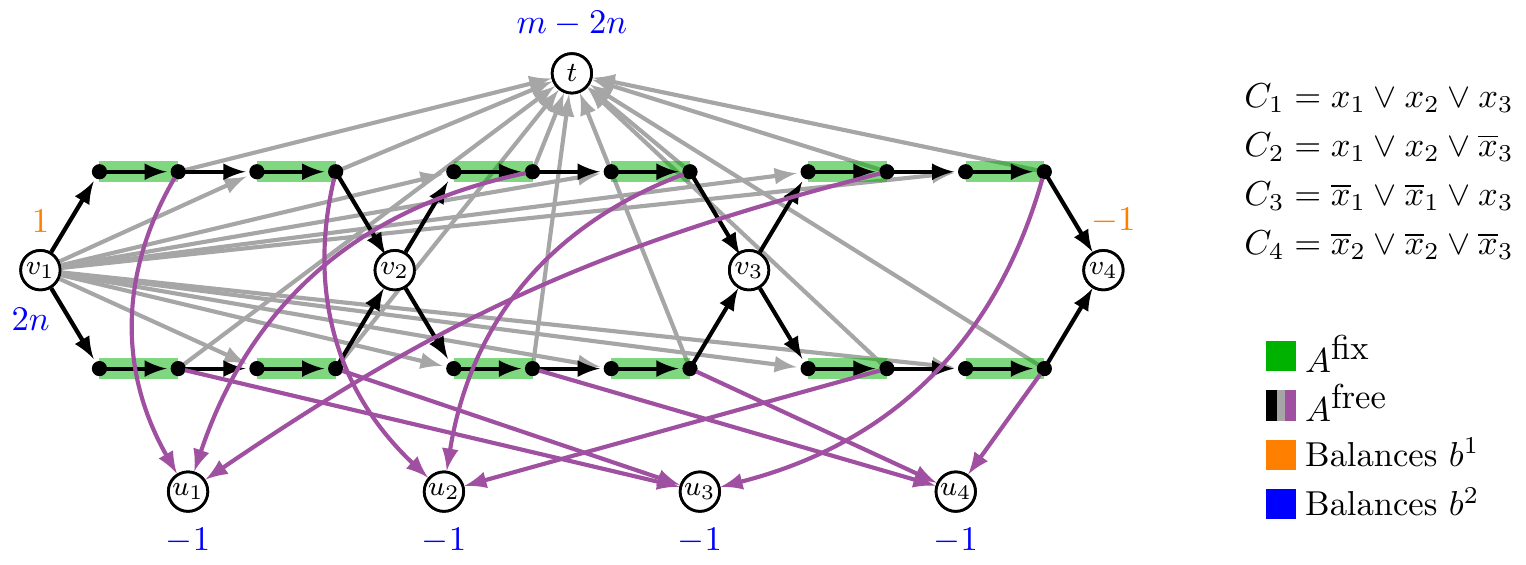}
	\end{adjustbox}
	\caption{Construction of the \ProblemName{} instance $\widetilde{\mathcal{I}}$}
	\label{Reduction}
\end{figure}
In general, the instance is based on a digraph $G=(V,A)$ defined as follows.
The vertex set $V$ consists of one vertex $v_i$ per variable $x_i$, $i\in[n]$, an additional dummy vertex $v_{n+1}$, and one vertex $u_j$ per clause $C_j$, $j\in[m]$. 
In addition, for every literal $x_i$ ($\overline{x}_i$), $i\in[n]$ four auxiliary vertices $w^\ell_i$ ($\overline{w}^\ell_i$), $\ell\in[4]$ are included as well as a further auxiliary vertex $t$.
Arc set $A$ includes arcs that connect two successive variable vertices $v_i$, $v_{i+1}$, $i\in[n]$ by two parallel paths $p_i$ and $\overline{p}_i$ defined along the auxiliary vertices, i.e., $p_i=v_i w^1_i w^2_i w^3_i w^4_i v_{i+1}$ and $\overline{p}_i=v_i\overline{w}^1_i \overline{w}^2_i \overline{w}^3_i \overline{w}^4_i v_{i+1}$ for $i\in[n]$. 
Path $p_i$ represents the positive literal $x_i$, and path $\overline{p}_i$ the negative literal $\overline{x}_i$ of instance $\mathcal{I}$.
As each literal occurs exactly twice in instance $\mathcal{I}$, we identify two arcs of paths $p_i$ and $\overline{p}_i$ each with the literals.
More precisely, let $x^k_i$ ($\overline{x}^k_i$) denote literal $x_i$ ($\overline{x}_i$) which occurs the $k$-th time, $k\in [2]$ in the formula. 
Arc $(w^{2k-1}_i,w^{2k}_i)$ ($(\overline{w}^{2k-1}_i,\overline{w}^{2k}_i)$), which we refer to as \textit{literal arc}, is supposed to correspond to literal $x^k_i$ ($\overline{x}^k_i$). 
Using this correspondence, we add arc $(w^{2k}_i, u_{j})$ ($(\overline{w}^{2k}_i, u_{j})$) for every literal $x^k_i$ ($\overline{x}^k_i$), $i\in[n]$, $k\in [2]$ included in clause $C_{j}$, $j\in[m]$. 
Finally, arcs $(v_1,w^\ell_i)$ and $(v_1,\overline{w}^\ell_i)$ for $\ell\in\{1,3\}$ as well as arcs $(w^\ell_i,t)$ and $(\overline{w}^\ell_i,t)$ for $\ell\in\{2,4\}$ are added for every $i\in[n]$.   

The fixed arcs of set $A$ are defined by all literal arcs, i.e., $A^\text{fix}= \big\{(w^\ell_i ,w^{\ell+1}_i), \linebreak[3] (\overline{w}^\ell_i , \overline{w}^{\ell+1}_i) \mid \ell\in \{1,3\},\ i\in [n]\big\},$ while all remaining arcs are contained in set $A^\text{free}$.
We set the capacity and cost to $u\equiv1$ and $c\equiv0$, respectively.
To conclude, we define balances $\boldsymbol{b}=(b^1,b^2)$ such that the unique source is given by vertex $v_1$. 
In contrast, depending on the scenario considered, vertex $v_{n+1}$, or vertices $u_1,\ldots,u_m$ and $t$ function as sinks. 
More precisely, we obtain
\begin{align*}
b^1(v)
= \left\{ \begin{array}{ll}
1 			& \mbox{if $v=v_1$}, \\ 
-1 			& \mbox{if $v=v_{n+1}$}, \\ 
0 			& \mbox{otherwise},  \end{array} \right.
& \ 
b^2(v)= \left\{ \begin{array}{ll}
2n 			& \mbox{if $v=v_1$}, \\ 
-1 			& \mbox{if $v=u_j, \ j\in [m]$}, \\
m-2n 	& \mbox{if $v=t$}, \\
0 			& \mbox{otherwise}.  \end{array} \right.
\end{align*}
In summary, we obtain a feasible \ProblemName{} instance $\widetilde{\mathcal{I}}= (G, u, c, \boldsymbol{b})$ that can be constructed in polynomial time. Hence, it remains to show that $\mathcal{I}$ is a Yes-instance if and only if for instance $\widetilde{\mathcal{I}}$ a robust $\boldsymbol{b}$-flow exists with cost at most $\beta:=0$.

For this purpose, let $x_1,\ldots,x_n$ be a satisfying truth assignment for instance $\mathcal{I}$. We define the first scenario flow $f^1$ of instance $\widetilde{\mathcal{I}}$ as follows
\begin{align*}
	f^1(a)=\left\{ \begin{array}{ll}
			 1 &\text{for all } a\in A(p_i)	 \mbox{ if $x_i=\textsc{True}$}, \\
			 1 &\text{for all } a\in A(\overline{p}_i)	 \mbox{ if $x_i=\textsc{False}$}, \\
				0  &\mbox{otherwise},  \end{array} \right.
\end{align*}
i.e., flow $f^1$ sends exactly one unit from source $v_1$ to sink $v_{n+1}$ using $2n$ literal arcs by sending flow along either path $p_i$ or $\overline{p}_i$, $i\in [n]$. 
As $x_1,\ldots,x_n$ is a satisfying truth assignment, there exists one verifying literal $x^k_i$ or $\overline{x}^k_i$, $k\in [2]$, $i\in [n]$ for each clause $C_j$, $j\in [m]$. 
We define the second scenario flow $f^2$ from the source to the clause vertices along the literal arcs which correspond to these verifying literals: 
\begin{align*}
f^2(a)=
\begin{cases}
1 		&\text{for all } a\in A(q) \text{ with }q=v_1w^{2k-1}_i w^{2k}_iu_{j} \mbox{ if $x^k_i\in C_{j}$ is verifying}, \\
1 		&\text{for all } a\in A(q) \text{ with }q=v_1\overline{w}^{2k-1}_i\overline{w}^{2k}_i u_{j} \mbox{ if $\overline{x}^k_i\in C_{j}$ is verifying},\\
0 		 &\mbox{otherwise}. 
\end{cases}
\end{align*}   
To satisfy the remaining $2n-m$ demand, flow $f^2$ is defined along the remaining, and also from flow $f^1$ used, literal arcs to sink $t$, i.e., $f^2(a)=1$ for all 
\begin{align*}
\left\{ \begin{array}{ll}
a\in A(p) \text{ with }p=v_1w^{2k-1}_iw^{2k}_i t \mbox{ if $x_i=\textsc{True}$ and $f^2((w^{2k-1}_i,w^{2k}))=0$}, \\
a\in A(p) \text{ with }p=v_1\overline{w}^{2k-1}_i\overline{w}^{2k}_i t	\mbox{ if $x_i=\textsc{False}$ and $f^2((\overline{w}^{2k-1}_i,\overline{w}^{2k}))=0$}.  \\
\end{array} \right.
\end{align*}
Overall, flow $f^2$ sends $2n-m$ units to sink $t$ and one unit to each of the sinks $u_1,\ldots,u_m$ using exactly $2n$ literal arcs. Consequently, we have constructed $b^\lambda$-flows $f^\lambda$ for both scenarios $\lambda\in \Lambda$ such that the consistent flow constraints are satisfied, and this results in a robust $\boldsymbol{b}$-flow $\boldsymbol{f}=(f^1,f^2)$ with cost $c(\boldsymbol{f})=0$. 

Conversely, let $\boldsymbol{f}=(f^1,f^2)$ be a robust $\boldsymbol{b}$-flow with at most zero cost.
Flow $f^2$ sends in total $2n$ units from vertex $v_1$ to vertices $u_1,\ldots,u_m$ and $t$. 
By construction of the network, the only option to reach each of these sinks requires the usage of at least one of the fixed literal arcs. 
Due to the integral flow $f^1$ sending only one unit within the acyclic digraph, it holds $f^1(a)= f^2(a)\in \{0,1\}$ for all fixed arcs $a\in A^\text{fix} $. 
Consequently, flows $f^1$ and $f^2$ use at least $2n$ fixed arcs in order to meet the demand of flow $f^2$.
As further consequence of the acyclic digraph, flow $f^1$ uses either path $p_i$ or $\overline{p}_i$, $i\in[n]$ but never both simultaneously to reach sink $v_{n+1}$.
Accordingly, if flow $f^1$ sends flow along path $p_i$, $i\in[n]$, we set $x_i=\textsc{True}$. 
If flow $f^1$ sends flow along path $\overline{p}_i$ for $i\in[n]$, our choice is $x_i=\textsc{False}$.  
Further, to meet the demand at sinks $u_j$, $j\in [m]$, flow $f^2$ sends flow along either subpath $w^\ell_iw^{\ell+1}_iu_{j}$ or $\overline{w}^\ell_i\overline{w}^{\ell+1}_iu_{j}$ for $\ell\in \{1,3\}$, $i\in [n]$, $j\in[m]$ but never both simultaneously.
In the former case, clause $C_{j}$ is verified due to the previous assignment $x_i=\textsc{True}$ induced by flow $f^1$ and the fact that $x_i\in C_{j}$ holds. 
In the latter case, clause $C_{j}$ is verified due to the included variable $x_i$ set to \textsc{False}. 
As a result, $x_1,\ldots,x_n$ is a satisfying truth assignment for instance $\mathcal{I}$.
\end{proof}
The statement of Theorem~\ref{Theorem:ReductionAcyclicGraphs} can be formulated even stronger as the following theorem shows.
\begin{theorem}\label{Theorem:ReductionAcyclicGraphs2}
The decision version of the \ProblemName{} problem for instances based on acyclic digraphs is strongly $\mathcal{NP}$-complete, even if only two scenarios on a network with a unique source and a unique sink are considered.
\end{theorem}
\begin{proof}
We extend the construction of proof of Theorem~\ref{Theorem:ReductionAcyclicGraphs} by free arcs $(u_j,v_{n+1})$ for $j\in [m]$ as well as the two free arcs $(v_1,w^4_n)$ and $(v_1,\overline{w}^4_n)$. 
Like all other arcs in the network, their capacities are set to one. 
The balances are updated such that $v_1$ serves as unique source and $v_{n+1}$ as unique sink. 
However, in the second scenario we  require $2n+2$ instead of $2n$ demand sent from the source to the sink.
This adjustment is necessary as we need to ensure that sufficient demand is sent along the clause vertices which are no sinks anymore.
Otherwise, there might exist a feasible robust $\boldsymbol{b}$-flow that sends a unit along path $v_1w^3_nw^4_nv_{n+1}$ or $v_1\overline{w}^3_n\overline{w}^4_nv_{n+1}$ which in turn allows one unsatisfied clause. 
Analogous to the proof of Theorem~\ref{Theorem:ReductionAcyclicGraphs}, a Yes-instance of a $(3,B2)$-\textsc{Sat} problem is equivalent to a Yes-instance of the \ProblemName{} problem.
\end{proof}

\section{\ProblemName{} Problem on Series-Parallel Digraphs}	
\label{Sec:SPDef}
In this section, we consider the \ProblemName{} problem on series-parallel (SP) digraphs.
We firstly propose a definition of SP digraphs and its representation in the form of a rooted binary decomposition tree. 
In Section~\ref{Subsec:SPComplexity+DP}, we show the weak $\mathcal{NP}$-completeness for networks with multiple sources and multiple sinks.
For the special case of networks with a unique source and a unique sink, we provide an algorithm which runs in polynomial time in Section~\ref{Subsec:SpecialCaseUniqueSourceUniqueSinkCapacities}.
\\\\
We start with a formal definition for SP digraphs based on the edge SP multi-graphs definition of Valdes et al.~(\citeyear{valdes1982recognition}).
\begin{definition}\label{Def:SeriesParallelGraphs}
\textit{Series-parallel (SP) digraphs} can be recursively defined as follows.
\begin{itemize}
\item[1.] An arc $(\sourceSPGraph,\sinkSPGraph)$ is an SP digraph with \textit{origin} $\sourceSPGraph$ and \textit{target} $\sinkSPGraph$.
\item[2.] Let $G_1$ with origin $\sourceSPGraph_1$ and target $\sinkSPGraph_1$ and $G_2$ with origin $\sourceSPGraph_2$ and target $\sinkSPGraph_2$ be SP digraphs. 
The digraph that is constructed by one of the following two compositions of SP digraphs $G_1$ and $G_2$ is itself an SP digraph.
\begin{itemize}
\item[a)] The \textit{series composition} $G$ of two SP digraphs $G_1$ and $G_2$ is the digraph obtained by contracting target $\sinkSPGraph_1$ and origin $\sourceSPGraph_2$. The origin of digraph $G$ is then $\sourceSPGraph_1$ (becoming $\sourceSPGraph$), and the target is $\sinkSPGraph_2$ (becoming $\sinkSPGraph$).
\item[b)] The \textit{parallel composition} $G$ of two SP digraphs $G_1$ and $G_2$ is the digraph obtained by contracting origins $\sourceSPGraph_1$ and $\sourceSPGraph_2$ (becoming $\sourceSPGraph$) and contracting targets $\sinkSPGraph_1$ and $\sinkSPGraph_2$ (becoming $\sinkSPGraph$). The new origin of digraph $G$ is $\sourceSPGraph$, and the target is $\sinkSPGraph$.
\end{itemize}
\end{itemize}
\end{definition}
The parallel and series compositions are illustrated in Figure~\ref{fig:Spgraphs}. 
Note that, by definition, SP digraphs are generally multi-graphs with one definite origin and one definite target. 

\begin{figure*}
	\begin{adjustbox}{max width=1\textwidth, max height=1\textheight}
		\subfloat[]{\label{fig:Spgraphs}	\includegraphics{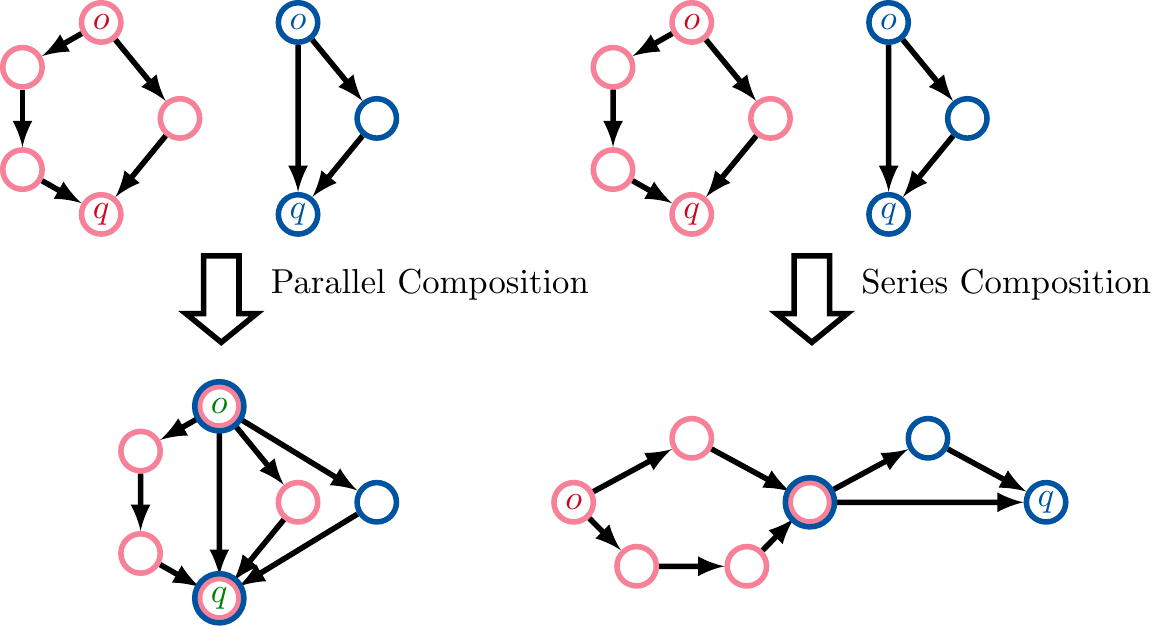}}
		\subfloat[]{\label{fig:Spgraphs2}	\includegraphics{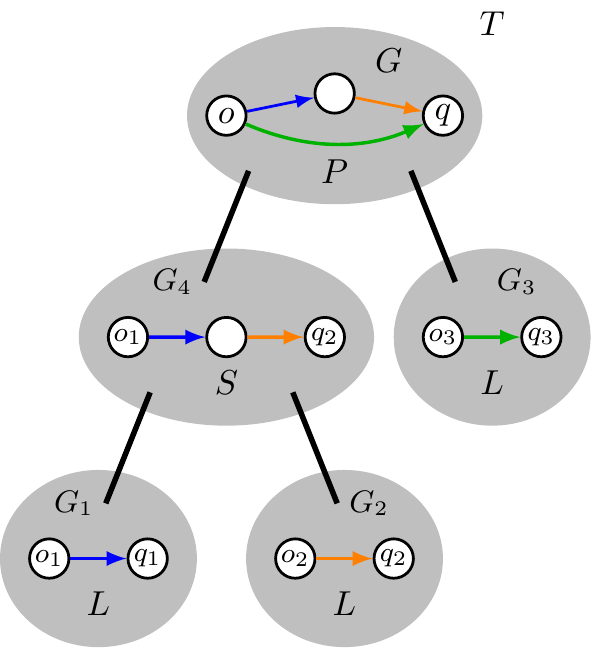}}
	\end{adjustbox}
\caption{(a) Example of an SP digraph defined by a parallel or series composition (b) Example of the representation of an SP digraph by its SP tree}
\end{figure*}
A useful property of SP digraphs is their representability in the form of a rooted binary decomposition tree, a so-called \textit{SP tree}, visualized in Figure~\ref{fig:Spgraphs2}. 
For a given SP digraph, we construct an SP tree that represents the order of the series and parallel compositions of individual arcs. 
By means of three different vertices, namely $L$-vertices, $S$-vertices, and $P$-vertices, single arcs as well as series and parallel compositions are indicated. 
The SP tree's leaves are $L$-vertices where there exist as many $L$-vertices as the digraph represented has arcs. 
The $S$- and $P$-vertices are the SP tree's inner vertices and correspond to the digraph obtained by a series or, respectively, parallel composition of the
subgraphs associated with their two child vertices. 
The order of the children of $P$-vertices is irrelevant while it is essential for $S$-vertices as the series composition is not commutative. 
Following the constructions of all series and parallel compositions, we obtain the entire digraph represented by the SP tree's root. 
The representation of an SP digraph by its SP tree can be beneficially used as the construction is conducted in polynomial time~\citep{Valdes:1979:RSP:800135.804393}.

\subsection{Multiple Sources and Multiple Sinks Networks}
\label{Subsec:SPComplexity+DP}
In this section, we firstly concentrate on the complexity of the \ProblemName{} problem for networks based on SP digraphs with a unique source and single sinks (but not a unique sink).
Afterwards, we conclude the complexity for networks with multiple sources and multiple sinks. 
We perform the reduction from \textsc{Partition} which is known to be weakly $\mathcal{NP}$-complete~\citep{johnson1979computers}. 
\begin{definition}[\textsc{Partition}]
Let $S=\{s_1,\ldots,s_n\}$ be a set of $n$ positive integers that sum up to $2w$, i.e., $\sum^{n}_{i=1}s_i=2w$. The decision problem of \textsc{Partition} asks whether there exists a partition of set $S$ in two disjoint subsets $S_1$ and $S_2$ such that the sum of the integers of subsets $S_1$ is equal to the sum of the integers included in subset $S_2$, i.e., $S=S_1\uplus S_2$ with
\begin{align*}
	\sum_{s_i\in S_1} s_i = \sum_{s_i\in S_2} s_i = w.
\end{align*} 
\end{definition}
\begin{theorem}
The decision version of the \ProblemName{} problem on networks based on SP digraphs with a unique source and single sinks is weakly $\mathcal{NP}$-complete. \label{Theorem:SPComplexity}
\end{theorem}
\begin{proof}
	Let $\mathcal{I}$ be a \textsc{Partition} instance with positive integers $s_1,\ldots,s_n$ such that $\sum^{n}_{i=1}s_i=2w$ holds. 
	We construct a corresponding \ProblemName{} instance $\widetilde{\mathcal{I}}= (G, u, c, \boldsymbol{b})$ considering a set of two scenarios, i.e., $\Lambda=\{1,2\}$, visualized in Figure~\ref{ReductionSeriesParallelGraphsPartition}.
	\begin{figure}
		\begin{adjustbox}{max width=\textwidth}
			\includegraphics{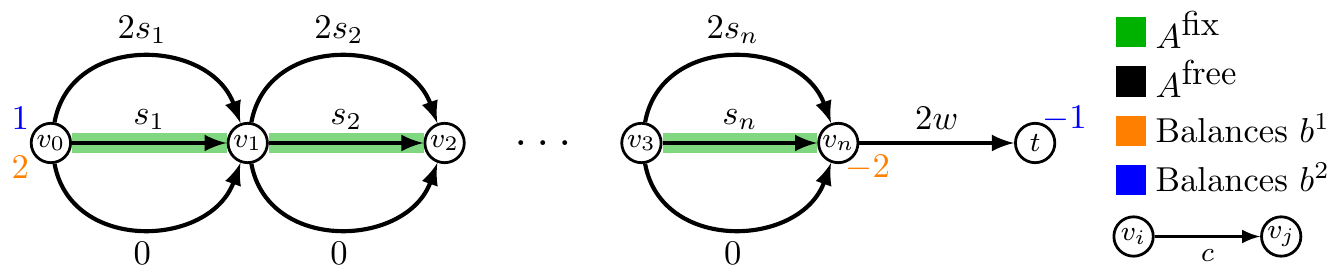}
		\end{adjustbox}
		\caption{Construction of the \ProblemName{} instance $\widetilde{\mathcal{I}}$}
		\label{ReductionSeriesParallelGraphsPartition}
	\end{figure}
	The network is based on an SP digraph $G=(V,A)$ where vertex set $V$ consists of two auxiliary vertices $v_0$ and $t$, and one vertex $v_i$ per integer $s_i$, $i\in[n]$.
	Arc set $A$ consists of multi-arcs $a_{i}$ for $i\in[n]$ that connects two successive vertices $v_{i-1}$, $v_{i}$ by three parallel arcs $a^1_{i}, a^2_{i},a^3_{i}$, plus a single arc $a_{n+1}$ from vertex $v_{n}$ to vertex $t$. 
	Multi-arc $a_i$, $i\in [n]$ is supposed to represent integer $s_i$, which is why we refer to as \textit{integer multi-arcs}.
	The fixed arcs of set $A$ are defined by the second arc of all integer multi-arcs each, i.e., $A^\text{fix}= \{a^2_i=(v_{i-1},v_{i}) \mid i\in [n]\}$, while all remaining arcs are contained in set $\free$.

	Further, we set the capacity on all arcs to one, i.e., $u\equiv1$.
	The cost $c$ is given such that the use of the first arc $a^1_i$ and second arc $a^2_i$ of every integer multi-arc costs two and one times the corresponding integer value $s_i$ per flow unit, respectively. 
	In turn, using the third arc $a^3_i$ causes zero cost.
	The use of arc $a_{n+1}$ costs $2w$ per flow unit. 
	To conclude, we define balances $\boldsymbol{b}=(b^1,b^2)$ on vertex set $V$ such that in the first scenario vertex $v_0$ supplies and vertex $v_n$ demands two units. 
	In the second scenario, vertex $v_0$ supplies and vertex $t$ demands one unit. 
	In both scenarios, the balances of all other vertices are equal to zero, i.e., overall we obtain
	\begin{align*}
	b^1(v)= \left\{ \begin{array}{ll}
	2 & \mbox{if $v=v_0$}, \\ 
	-2 & \mbox{if $v=v_n$}, \\ 
	0 & \mbox{otherwise},  \end{array} \right.
	\
	& b^2(v)= \left\{ \begin{array}{ll}
	1 & \mbox{if $v=v_0$}, \\ 
	-1 & \mbox{if $v=t$}, \\
	0 & \mbox{otherwise}.  \end{array} \right.
	\end{align*}
	Accordingly, for both scenarios the unique source is defined by vertex $v_0$ while depending on the scenario considered vertex $v_n$ or $t$ serves as single sink. 
	In order to satisfy demand with supply, flow is sent along paths through the network.
	For convenience, let $p^\ell$ denote the path along the $\ell$-th integer multi-arcs for $\ell\in [3]$, i.e., $p^\ell=a^\ell_1\ldots a^\ell_{n}$.
	Overall, we obtain a feasible \ProblemName{} instance $\widetilde{\mathcal{I}}= (G, u, c, \boldsymbol{b})$ that can be constructed in polynomial time. 
	Hence, it remains to show that $\mathcal{I}$ is a Yes-instance if and only if for instance $\widetilde{\mathcal{I}}$ a robust $\boldsymbol{b}$-flow exists with cost of at most $\beta:=3w$.
	
	For this purpose, let $S_1$ and $S_2$ be a feasible partition for instance $\mathcal{I}$.
	We define the first scenario flow $f^1$ of instance $\widetilde{\mathcal{I}}$ by 
	\begin{align*}
	f^1(a)= \left\{ \begin{array}{ll}
	1 & \text{for all arcs }a=a^1_i\in A	 \mbox{ if $s_i\in S_1$}, \\
	1 & \text{for all arcs }a=a^2_i\in A	 \mbox{ if $s_i\in S_2$}, 	\\
	1 & \text{for all arcs }a=a^3_i\in A, 	\\
	0 & \text{otherwise}, \end{array} \right.
	\end{align*}
	i.e., flow $f^1$ sends one unit from source $v_0$ to sink $v_n$ along arcs of paths $p^1$ and $p^2$, while one further unit is sent using arcs of path $p^3$ only. 
	As the sets $S_1$ and $S_2$ form a feasible partition we obtain cost of 
	\begin{align*}
	c(f^1)	
	&= \sum_{a\in A} c(a) f^1(a)= \sum_{a\in \fix} c(a) f^1(a) + \sum_{a\in \free}  c(a) f^1(a)   =w+2w =3w.
	\end{align*}
	According to flow $f^1$ and the partition, we define the second scenario flow $f^2$ by
	\begin{align*}
	f^2(a)= \left\{ \begin{array}{ll}
	1 & \text{for all arcs }a=a^2_i\in A	 \mbox{ if $s_i\in S_2$}, \\
	1 & \text{for all arcs }a=a^3_i\in A	 \mbox{ if $s_i\in S_1$}, 	\\
	1 & \text{for arc }a=a_{n+1}\in A, 	\\
	0 & \text{otherwise}, \end{array} \right.
	\end{align*}
	i.e., flow $f^2$ sends exactly one unit from source $v_0$ to sink $t$ along arcs of paths $p^2$ and $p^3$, and by using arc $a_{n+1}$. 
	The following cost is caused
	\begin{align*}
	c(f^2)	
	&= \sum_{a\in A} c(a) f^2(a)\\
	&= \sum_{a\in A^\text{fix}} c(a) f^2(a) + \sum_{a\in A^\text{free}\setminus\{a_{n+1}\}}  c(a) f^2(a) + c(a_{n+1}) f^2(a_{n+1})   \\
	&=w+0+2w =3w.
	\end{align*}
	Consequently, we have constructed a robust $\boldsymbol{b}$-flow $\boldsymbol{f}=(f^1,f^2)$ with cost of $3w$. 
		
	Conversely, let $\boldsymbol{f}=(f^1,f^2)$ be a robust $\boldsymbol{b}$-flow with cost of at most $3w$, i.e., $ c(\boldsymbol{f})=\max \{c(f^1), c(f^2)\}\leq 3w.$
	The first scenario flow $f^1$ sends two units from source $v_0$ to sink $v_n$. 
	Due to the capacities, not only a single path is used to send these flow units. 
	In particular, not only path $p^3$ causing zero cost is used. 
	As sending one flow unit along path $p^1$ would cause cost of 
	$$c(p^1)= \sum_{a\in A(p^1)}c(a)=\sum^n_{i=1}2 s_i = 4w > 3w=\max\{c(f^1), c(f^2)\},$$ 
	flow $f^1$ does not use all arcs of path $p^1$ either. 
	Accordingly, flow $f^1$ uses as many arcs of path $p^2$ as at least cost of $w$ is caused in order that at most $2w$ cost is caused due to arcs of path $p^1$.  
	
	The second scenario flow $f^2$ sends one unit from source $v_0$ to sink $t$. 
	As flow $f^2$ uses arc $a_{n}$ with cost of $2w$ to reach sink $t$, the unit is sent along arcs of paths $p^1$, $p^2$, $p^3$ such that at most cost of $w$ is caused.
	Furthermore, as flow $f^1$ uses as many arcs of path $p^2$ as at least cost of $w$ is caused, this also holds true for flow $f^2$ as $A(p^2)=\fix$ holds.
	Consequently, flow $f^2$ only uses arcs of path $p^2$ and $p^3$, however, due to the acyclic SP digraph never of the same multi-arc simultaneously such that the sets
	\begin{align*}
	S_1&:=\{s_i\mid f^2(a^2_i)=1 \text{ for arc }a^2_i\in A^\text{fix} \text{ with }i\in [n]\},\\
	S_2&:=\{s_i\mid f^2(a^3_i)=1 \text{ for arc }a^3_i\in A(p^3)\text{ with }i\in[n]\}
	\end{align*} form a feasible partition for instance $\mathcal{I}$. 
\end{proof}
In the next step, we refute the strong $\mathcal{NP}$-completeness.
Therefore, we propose a pseudo-polynomial algorithm based on dynamic programming.
The dynamic program (DP) is applicable for networks with an arbitrary number of sources and sinks, especially for multiple sources and multiple sinks.  
The core idea of the DP is a bottom-up method using the SP tree.
While composing the SP digraph step by step, in each of these steps a robust flow is sought satisfying additional restrictions explained in the following.
The flow needs to send a given supply from the origin through the subgraph considered in the current step. Further, the flow needs to satisfy the inner vertices' balances are satisfied as their in- and outgoing arcs are already set.
In contrast, the balances at the origin and target do not have to be satisfied as in subsequent steps further subgraphs can still be composed at these vertices.
Moreover, the  flow must exactly meet a budget given.
Backtracking the steps of the DP results to an optimal robust flow. 
Before we present the DP in more detail, we introduce the notations and labels needed.

Let us consider a \ProblemName{} instance $(G,u,c,\boldsymbol{b})$ where $G$ is an SP digraph with origin $\sourceSPGraph$ and target $\sinkSPGraph$.
Further, let $T$ be the SP tree of digraph $G$ with its root vertex $r\in V(T)$. 
We denote the subgraph of digraph $G$ that is associated to vertex $v\in V(T)$ by $G_v$, and its origin and target by $\sourceSPGraph_v$ and $\sinkSPGraph_v$, respectively.
The algorithm relies on demand labels $d_v(\boldsymbol{\balanceDynPro_{v}},\boldsymbol{\costDynPro_v})$ defined for every subgraph $G_v$ associated with a vertex $v\in V(T)$. 
The parameter vector $\boldsymbol{\balanceDynPro_{v}}=(\balanceDynPro^1_{v}, \ldots,\balanceDynPro_{v}^{|\Lambda|})\in \mathbb{Z}^{|\Lambda|}_{\geq 0}$ determines for all scenarios the supply at origin $o_v$ of subgraph $G_v$. 
For every scenario $\lambda\in \Lambda$ the supply $\balanceDynPro^\lambda_{v}$ is limited by the sum of the capacities of all outgoing arcs of origin $o_v$ of subgraph $G_v$, i.e., $\balanceDynPro^\lambda_{v}\in \{0,\ldots, U_{v}\}$ with $U_{v}=\sum_{ a=(\sourceSPGraph_v,w)\in A(G_v)} u(a)$. 
The parameter vector $\boldsymbol{\costDynPro_v}=(\costDynPro_{v}^1,\ldots,\costDynPro_{v}^{|\Lambda|})\in \mathbb{Z}^{|\Lambda|}_{\geq 0}$ specifies for all scenarios the budget that must be spent for sending the supply in subgraph $G_v$ with respect to cost function $c$. 
Consequently, an upper bound on the budget is given by the cost that may occur in subgraph $G_v$, i.e., $\costDynPro^\lambda_{v}\in \{0,\ldots, C_{v}\}$ for $\lambda\in \Lambda$ with $C_{v}=\sum_{a\in A(G_v)} c(a)\cdot u(a)$. 

Let the \textit{$(\boldsymbol{\balanceDynPro_{v}}$, $\boldsymbol{\costDynPro_v})$-restricted robust minimum cost flow problem under consistent flow constraints} (\RestrictedProblemName{v}) be defined as the \ProblemName{} problem on subgraph $G_v$, $v\in V(T)$ with restrictions implied by supply $\boldsymbol{\balanceDynPro_{v}}$ and budget $\boldsymbol{\costDynPro_v}$.
The demand label $d_v(\boldsymbol{\balanceDynPro_{v}},\boldsymbol{\costDynPro_v})$ is defined as the optimal solution value of the \RestrictedProblemName{v} problem.
For convenience and for the sake of clarity, we indicate the \RestrictedProblemName{v} problem by the following integer program formulation.
\begin{align}
d_v(\boldsymbol{\balanceDynPro_{v}}, \boldsymbol{\costDynPro_v}) \nonumber 
\\ = \min \ \ &0   \label{DPZF}\\ 
\text{ s.t. } 
& \sum_{a\in A(G_v)} c(a)\cdot f^\lambda_a= \tilde{c}^\lambda_v &&\forall  \lambda \in \Lambda  \label{DPConstr1}\\
& \sum_{a=(w,z) \in A(G_v)}f^\lambda_a  - \sum_{a=(z,w) \in A(G_v)}f^\lambda_a \nonumber
\\&= 
\left\{ \begin{array}{ll}
b^\lambda(w) & \mbox{if $w\neq \sourceSPGraph_v$,}   \\ 
\balanceDynPro^\lambda_w & \mbox{if $w= \sourceSPGraph_v$},  \end{array} \right. 
&& \forall w\in V(G_v)\setminus\{\sinkSPGraph_v\},\lambda\in \Lambda \label{DPConstr2}\\
& f^\lambda_a=f^{\lambda^\prime}_a && \forall a \in A^{\text{fix}}(G_v), \lambda, \lambda^\prime\in \Lambda \label{DPConstr3}\\
& 0\leq f^\lambda_a \leq u(a) && \forall a \in A(G_v), \lambda\in \Lambda \label{DPConstr4}\\
& f^\lambda_a \in \mathbb{Z}_{\geq 0} && \forall a \in A(G_v), \lambda\in \Lambda \label{DPConstr5}
\end{align}
The \RestrictedProblemName{v} problem requires a robust $\boldsymbol{b}$-flow in subgraph $G_v$ by means of constraints \eqref{DPConstr2}-\eqref{DPConstr5}.
Therefore, the flow needs to satisfy the supply $\boldsymbol{\balanceDynPro_{v}}$ at origin $o_v$, and the balances $\boldsymbol{b}$ at all other vertices except target $\sinkSPGraph_v$.
Furthermore, the flow must exactly meet the budget $\boldsymbol{\costDynPro_v}$, see constraint \eqref{DPConstr1}. 
By definition of the objective function~\eqref{DPZF}, finding a feasible solution is sufficient to solve the \RestrictedProblemName{v} problem, i.e., $d_v(\boldsymbol{\balanceDynPro_{v}}, \boldsymbol{\costDynPro_v})\in \{0,\infty\}$.\\
 
For solving the \ProblemName{} problem on SP digraphs, the DP exploits the structure of the SP tree to compute demand labels recursively. 
More precisely, considering a specific vertex in the SP tree, we update the corresponding demand label based on the labels corresponding to the children's vertices in a bottom-up procedure. 
Depending on whether the SP tree's vertex considered is an $L$-, $S$-, or $P$-vertex, one of the following three procedures is applied.
We start with the initialization at the leaves.
 
\begin{lemma}\label{LemmaDPInitalisierung}
Let $v\in V(T)$ be a leaf of SP tree $T$, i.e., $v$ is an $L$-vertex. 
The demand label $d_v(\boldsymbol{\balanceDynPro_{v}}, \boldsymbol{\costDynPro_{v}})$ is initialized by
\begin{align*}
d_v(\boldsymbol{\balanceDynPro_{v}}, \boldsymbol{\costDynPro_v})
= \left\{ \begin{array}{ll}
0 & \mbox{if $(\sourceSPGraph_v,\sinkSPGraph_v)\in \free(G_v)$, $\costDynPro^\lambda_v =c((\sourceSPGraph_v, \sinkSPGraph_v))\cdot \balanceDynPro_{v}^\lambda$, $\forall\lambda\in \Lambda$}, \\ 
0 & \mbox{if $(\sourceSPGraph_v,\sinkSPGraph_v)\in \fix(G_v)$, $\balanceDynPro^\lambda_{v}= \balanceDynPro^{\lambda^\prime}_{v}$, $\costDynPro^\lambda_v = c((\sourceSPGraph_v, \sinkSPGraph_v))\cdot \balanceDynPro_{v}^\lambda$, $ \forall\lambda, \lambda^\prime\in \Lambda$}, \\ 
\infty & \mbox{otherwise}.  \end{array} \right.
\end{align*}
\end{lemma}
\begin{proof}
As $v\in V(T)$ is an $L$-vertex, subgraph $G_v$ only consists of the single arc $a_v:=(\sourceSPGraph_v,\sinkSPGraph_v)$. 
If $a_v$ is a free arc, i.e., $a_v\in \free(G_v)$, $\costDynPro^\lambda_v =c((\sourceSPGraph_v, \sinkSPGraph_v))\cdot \balanceDynPro_{v}^\lambda$ must hold true. 
Otherwise, there exists no feasible flow that satisfies constraints~\eqref{DPConstr1} due to constraints~\eqref{DPConstr2}.
Consequently, the \RestrictedProblemName{v} problem is not solvable, i.e., $d_v(\boldsymbol{\balanceDynPro_{v}}, \boldsymbol{\costDynPro_{v}})=\infty$. 
If $a_v$ is a fixed arc, i.e., $a_v\in \fix(G_v)$, the constraints of the previous case need to be satisfied due to the same argumentation. In addition, $\balanceDynPro^\lambda_{v}= \balanceDynPro^{\lambda^\prime}_{v}$ must hold true for all scenarios $\lambda, \lambda^\prime\in \Lambda$ by reason of constraints~\eqref{DPConstr3}.
To conclude, if the presented constraints are satisfied, an optimal solution to the \RestrictedProblemName{v} problem is given by $\boldsymbol{f}(a_v):=\boldsymbol{\balanceDynPro_v}$ such that $d_v(\boldsymbol{\balanceDynPro_{v}}, \boldsymbol{\costDynPro_{v}})=0$ holds true.   
\end{proof}
In the next step, we consider the case in which the demand label is derived recursively from the demand labels of the child vertices that are parallelly composed.
\begin{lemma}\label{LemmaDPParalleleKomposition}
Let $v\in V(T)$ be a $P$-vertex in SP tree $T$ with child vertices $x,y\in V(T)$. The demand label $d_v(\boldsymbol{\balanceDynPro_{v}}, \boldsymbol{\costDynPro_{v}})$ at vertex $v$ can be computed by a composition of the demand labels $d_x(\boldsymbol{\balanceDynPro_{x}}, \boldsymbol{\costDynPro_x})$ and $d_y(\boldsymbol{\balanceDynPro_{y}}, \boldsymbol{\costDynPro_y}) $ of its child vertices $x$ and $y$ as follows
\begin{align*}
d_v(\boldsymbol{\balanceDynPro_{v}}, \boldsymbol{\costDynPro_{v}})
= \min\limits_{\substack{
		\boldsymbol{\balanceDynPro_{v}} =\boldsymbol{\balanceDynPro_{x}}	+ \boldsymbol{\balanceDynPro_{y}}\\
		\boldsymbol{\costDynPro_{v}}    =\boldsymbol{\costDynPro_{x}} 		+ \boldsymbol{\costDynPro_{y}}
}}
\left\{    
d_x(\boldsymbol{\balanceDynPro_{x}}, \boldsymbol{\costDynPro_x})  + d_y(\boldsymbol{\balanceDynPro_{y}}, \boldsymbol{\costDynPro_y})  
\right\}. 
\end{align*} 
\end{lemma}
\begin{proof}
For vertex $v\in V(T)$, let $d_v(\boldsymbol{\balanceDynPro_{v}}, \boldsymbol{\costDynPro_{v}})$ be the demand label with the related solution $\boldsymbol{f^*}$. 
As $v$ is a $P$-vertex, flow $\boldsymbol{f^*}$ with associated supply $\boldsymbol{\balanceDynPro_{v}}=\sum_{a=(\sourceSPGraph_v,w)\in A(G_v)} \boldsymbol{f^*}(a)$ can be divided into two flows $\boldsymbol{f_x}$ and $\boldsymbol{f_y}$ with associated supplies $\boldsymbol{\balanceDynPro_{x}}$ and $\boldsymbol{\balanceDynPro_{y}}$, respectively.
Flow $\boldsymbol{f_x}$ is defined on subgraph $G_x$, and flow $\boldsymbol{f_y}$ is defined on subgraph $G_y$ only. 
The budget $\boldsymbol{\costDynPro_{v}}=\sum_{a\in A(G_v)} c(a)\cdot \boldsymbol{f^*}(a)$ of flow $\boldsymbol{f^*}$ is also divided such that $\boldsymbol{\costDynPro_{x}}$ describes the budget of flow $\boldsymbol{f_x}$ and $\boldsymbol{\costDynPro_{y}}$ the budget of flow $\boldsymbol{f_y}$. 
Flows $\boldsymbol{f_x}$ and $\boldsymbol{f_y}$ are feasible solutions to the \RestrictedProblemName{x} and \RestrictedProblemName{y} problem, respectively.
Consequently, we obtain 
\begin{align*}
d_v(\boldsymbol{\balanceDynPro_{v}}, \boldsymbol{\costDynPro_{v}})
= 
d_v(\boldsymbol{\balanceDynPro_{x}}+\boldsymbol{\balanceDynPro_{y}}, \boldsymbol{\costDynPro_x}+\boldsymbol{\costDynPro_y})  
&\geq
d_x(\boldsymbol{\balanceDynPro_{x}}, \boldsymbol{\costDynPro_x})  + d_y(\boldsymbol{\balanceDynPro_{y}}, \boldsymbol{\costDynPro_y}),
\end{align*}
where $d_x(\boldsymbol{\balanceDynPro_{x}}, \boldsymbol{\costDynPro_{x}})$ and $d_y(\boldsymbol{\balanceDynPro_{y}}, \boldsymbol{\costDynPro_{y}})$ are the demand labels corresponding to the child vertices $x,y\in V(T)$.
In particular, this implies 
\begin{align*}
d_v(\boldsymbol{\balanceDynPro_{v}}, \boldsymbol{\costDynPro_{v}})
\geq
\min\limits_{\substack{
		\boldsymbol{\balanceDynPro_{v}} =\boldsymbol{\balanceDynPro_{x}}	+ \boldsymbol{\balanceDynPro_{y}}\\
		\boldsymbol{\costDynPro_{v}}    =\boldsymbol{\costDynPro_{x}} 		+ \boldsymbol{\costDynPro_{y}}
}}
\left\{    
d_x(\boldsymbol{\balanceDynPro_{x}}, \boldsymbol{\costDynPro_x})  + d_y(\boldsymbol{\balanceDynPro_{y}}, \boldsymbol{\costDynPro_y})  
\right\}.
\end{align*}

Conversely, for child vertices $x,y\in V(T)$, let $d_x(\boldsymbol{\balanceDynPro_{x}}, \boldsymbol{\costDynPro_{x}})$ and $d_y(\boldsymbol{\balanceDynPro_{y}}, \boldsymbol{\costDynPro_{y}})$ be the demand labels with related solutions $\boldsymbol{f^*_x}$ and $\boldsymbol{f^*_y}$. 
Combining flows $\boldsymbol{f^*_x}$ and $\boldsymbol{f^*_y}$ results in a feasible solution $\boldsymbol{f_{v}}:=\boldsymbol{f^*_x}+ \boldsymbol{f^*_y}$ to the \RestrictedProblemName{v} problem with supply $\boldsymbol{\balanceDynPro_v}:=\boldsymbol{\balanceDynPro_{x}}+\boldsymbol{\balanceDynPro_{y}}$ and budget $\boldsymbol{\costDynPro_{v}}:=\boldsymbol{\costDynPro_{x}}+\boldsymbol{\costDynPro_{y}}$. 
Consequently, for all supplies $\boldsymbol{\balanceDynPro_{x}}$, $\boldsymbol{\balanceDynPro_{y}}$ and budgets $\boldsymbol{\costDynPro_{x}}$, $\boldsymbol{\costDynPro_{y}}$ given the following holds true 
\begin{align*}
d_x(\boldsymbol{\balanceDynPro_{x}}, \boldsymbol{\costDynPro_x})  + d_y(\boldsymbol{\balanceDynPro_{y}}, \boldsymbol{\costDynPro_y})
\geq
 d_v(\boldsymbol{\balanceDynPro_{x}}+\boldsymbol{\balanceDynPro_{y}}, \boldsymbol{\costDynPro_x}+\boldsymbol{\costDynPro_y})
=d_v(\boldsymbol{\balanceDynPro_{v}}, \boldsymbol{\costDynPro_v}),
\end{align*}
where $d_v(\boldsymbol{\balanceDynPro_{v}}, \boldsymbol{\costDynPro_v})$ is the demand label corresponding to vertex $v\in V(T)$.
This implies
\begin{align*}
d_v(\boldsymbol{\balanceDynPro_{v}}, \boldsymbol{\costDynPro_v})  
\leq 
\min\limits_{\substack{
		\boldsymbol{\balanceDynPro_{v}} =\boldsymbol{\balanceDynPro_{x}}	+ \boldsymbol{\balanceDynPro_{y}}\\
		\boldsymbol{\costDynPro_{v}}    =\boldsymbol{\costDynPro_{x}} 		+ \boldsymbol{\costDynPro_{y}}
}}
\left\{    
d_x(\boldsymbol{\balanceDynPro_{x}}, \boldsymbol{\costDynPro_x})  + d_y(\boldsymbol{\balanceDynPro_{y}}, \boldsymbol{\costDynPro_y})  
\right\}.
\end{align*}
\end{proof}
To conclude the computation of demand labels, we consider the case in which a demand label is derived recursively from the demand labels of the child vertices that are serially composed.
\begin{lemma}\label{LemmaDPSerielleKomposition}
Let $v\in V(T)$ be an $S$-vertex in SP tree $T$ with child vertices $x,y\in V(T)$. 
The demand label $d_v(\boldsymbol{\balanceDynPro_{v}}, \boldsymbol{\costDynPro_{v}})$ at vertex $v$ can be computed by a composition of the demand labels $d_x(\boldsymbol{\balanceDynPro_{x}}, \boldsymbol{\costDynPro_x})$ and $d_y(\boldsymbol{\balanceDynPro_{y}}, \boldsymbol{\costDynPro_y}) $ of its child vertices $x$ and $y$ by
\begin{align*}
d_v(\boldsymbol{\balanceDynPro_{v}}, \boldsymbol{\costDynPro_{v}})
= \min\limits_{\substack{
		\boldsymbol{\balanceDynPro_{x}}= \boldsymbol{\balanceDynPro_{v}}\\
		\boldsymbol{\balanceDynPro_{y}}=  \boldsymbol{\balanceDynPro_{x}} + \boldsymbol{\beta} 	\\
		\boldsymbol{\costDynPro_{v}}   =\boldsymbol{\costDynPro_{x}} 		+ \boldsymbol{\costDynPro_{y}}
}}
\left\{    
d_x(\boldsymbol{\balanceDynPro_{x}}, \boldsymbol{\costDynPro_x})  + d_y(\boldsymbol{\balanceDynPro_{y}}, \boldsymbol{\costDynPro_y})  
\right\},
\end{align*}
where $\boldsymbol{\beta}=(\beta^1,\ldots,\beta^{|\Lambda|})$ with $ \beta^\lambda:= \sum_{v\in V(G_x)\setminus\{\sourceSPGraph_x\}}b^\lambda(v)$ holds for every scenario $\lambda\in \Lambda$. 
\end{lemma}
\begin{proof}
For vertex $v\in V(T)$, let $d_v(\boldsymbol{\balanceDynPro_{v}}, \boldsymbol{\costDynPro_{v}})$ be the demand label with the related solution $\boldsymbol{f^*}$. 
We assume that digraph $G_v$ is constructed by contracting the target $\sinkSPGraph_x$ of subgraph $G_x$ with the origin $\sourceSPGraph_y$ of subgraph $G_y$. 
Consequently, the flow that is sent through subgraph $G_y$ requires on the one hand the access via origin $\sourceSPGraph_y$. 
On the other hand, at least the same amount of flow is originated in subgraph $G_x$ in the first place.  
Using this insight, we partition flow $\boldsymbol{f^*}$ and the associated supply $\boldsymbol{\balanceDynPro_{v}}=\sum_{a=(\sourceSPGraph_v,w)\in A(G_v)} \boldsymbol{f^*}(a)$ in two flows $\boldsymbol{f_x}$ and $\boldsymbol{f_y}$ where flow $\boldsymbol{f_x}$ is defined on subgraph $G_x$ and flow $\boldsymbol{f_y}$ is defined on subgraph $G_y$ only. 
More precisely, we obtain $\boldsymbol{f_x}(a):=\boldsymbol{f^*}(a)$ for all arcs $a\in A(G_x)$ with associated supply $\boldsymbol{\balanceDynPro_{x}}=\boldsymbol{\balanceDynPro_{v}}$, and $\boldsymbol{f_y}(a):=\boldsymbol{f^*}(a)$ for all arcs $a\in A(G_y)$ with associated supply $\boldsymbol{\balanceDynPro_{y}}:=\sum_{a=(\sourceSPGraph_y,w)\in A(G_y)} \boldsymbol{f^*}(a)= \boldsymbol{\balanceDynPro_{x}}+\boldsymbol{\beta}$ where $\boldsymbol{\beta}=(\beta^1,\ldots,\beta^{|\Lambda|})$ with $\beta^\lambda:= \sum_{v\in V(G_x)\setminus\{\sourceSPGraph_x\}}b^\lambda(v)$, $\lambda\in \Lambda$. 
The associated supply $\boldsymbol{\balanceDynPro_{y}}$ results from the supply $\boldsymbol{\balanceDynPro_{x}}$ plus the flow that originates from sources (that are different from the origin) in subgraph $G_x$ minus the flow that is absorbed at sinks in subgraph $G_x$. 
The budget $\boldsymbol{\costDynPro_{v}}=\sum_{a\in A(G_v)} c(a)\cdot \boldsymbol{f^*}(a)$ of flow $\boldsymbol{f^*}$ can also be divided such that $\boldsymbol{\costDynPro_{x}}$ describes the budget of flow $\boldsymbol{f_x}$ and $\boldsymbol{\costDynPro_{y}}$ the one of flow $\boldsymbol{f_y}$.
Flows $\boldsymbol{f_x}$ and $\boldsymbol{f_y}$ are feasible solutions to the \RestrictedProblemName{x} and \RestrictedProblemName{y} problem, respectively. 
Consequently, we obtain 
\begin{align*}
d_v(\boldsymbol{\balanceDynPro_{v}}, \boldsymbol{\costDynPro_{v}})
= 
d_v(\boldsymbol{\balanceDynPro_{x}}, \boldsymbol{\costDynPro_x}+\boldsymbol{\costDynPro_y})  
&\geq
d_x(\boldsymbol{\balanceDynPro_{x}}, \boldsymbol{\costDynPro_x})  + d_y(\boldsymbol{\balanceDynPro_{y}}, \boldsymbol{\costDynPro_y}), 
\end{align*}
where $d_x(\boldsymbol{\balanceDynPro_{x}}, \boldsymbol{\costDynPro_x})$ and $d_y(\boldsymbol{\balanceDynPro_{y}}, \boldsymbol{\costDynPro_y})$ are the demand labels corresponding to child vertices $x,y\in V(T)$.
In particular, this implies
\begin{align*}
d_v(\boldsymbol{\balanceDynPro_{v}}, \boldsymbol{\costDynPro_{v}})
\geq
\min\limits_{\substack{
		\boldsymbol{\balanceDynPro_{x}}= \boldsymbol{\balanceDynPro_{v}}\\
		\boldsymbol{\balanceDynPro_{y}}=  \boldsymbol{\balanceDynPro_{x}} + \boldsymbol{\beta} 	\\
		\boldsymbol{\costDynPro_{v}}   =\boldsymbol{\costDynPro_{x}} 		+ \boldsymbol{\costDynPro_{y}}
}}
\left\{    
d_x(\boldsymbol{\balanceDynPro_{x}}, \boldsymbol{\costDynPro_x})  + d_y(\boldsymbol{\balanceDynPro_{y}}, \boldsymbol{\costDynPro_y})  
\right\}.
\end{align*}

Conversely, for child vertices $x,y\in V(T)$, let $d_x(\boldsymbol{\balanceDynPro_{x}}, \boldsymbol{\costDynPro_{x}})$ and $d_y(\boldsymbol{\balanceDynPro_{y}}, \boldsymbol{\costDynPro_{y}})$ with $\boldsymbol{\balanceDynPro_{y}}=\boldsymbol{\balanceDynPro_{x}} + \boldsymbol{\beta}$ be the demand labels with related solutions $\boldsymbol{f^*_x}$ and $\boldsymbol{f^*_y}$. 
Combining flows $\boldsymbol{f^*_x}$ and $\boldsymbol{f^*_y}$ results in a feasible solution $\boldsymbol{f_{v}}:=\boldsymbol{f^*_x}+\boldsymbol{f^*_y}$ to the \RestrictedProblemName{v} problem with supply $\boldsymbol{\balanceDynPro_v}:=\boldsymbol{\balanceDynPro_{x}}$ and budget $\boldsymbol{\costDynPro_{v}}:=\boldsymbol{\costDynPro_{x}}+\boldsymbol{\costDynPro_{y}}$.
Consequently, for all supplies $\boldsymbol{\balanceDynPro_{x}}$, $\boldsymbol{\balanceDynPro_{y}}$ and budgets $\boldsymbol{\costDynPro_{x}}$, $\boldsymbol{\costDynPro_{y}}$ the following holds true 
\begin{align*}
d_x(\boldsymbol{\balanceDynPro_{x}}, \boldsymbol{\costDynPro_x})  + d_y(\boldsymbol{\balanceDynPro_{y}}, \boldsymbol{\costDynPro_y})
\geq
d_v(\boldsymbol{\balanceDynPro_{x}}, \boldsymbol{\costDynPro_x}+\boldsymbol{\costDynPro_y})
=d_v(\boldsymbol{\balanceDynPro_{v}}, \boldsymbol{\costDynPro_v}),
\end{align*}
where $d_v(\boldsymbol{\balanceDynPro_{v}}, \boldsymbol{\costDynPro_v})$ is the demand label corresponding to vertex $v\in V(T)$.
This implies
\begin{align*}
d_v(\boldsymbol{\balanceDynPro_{v}}, \boldsymbol{\costDynPro_v})  
\leq 
\min\limits_{\substack{
		\boldsymbol{\balanceDynPro_{x}}= \boldsymbol{\balanceDynPro_{v}}\\
		\boldsymbol{\balanceDynPro_{y}}=  \boldsymbol{\balanceDynPro_{x}} + \boldsymbol{\beta} 	\\
		\boldsymbol{\costDynPro_{v}}   =\boldsymbol{\costDynPro_{x}} 		+ \boldsymbol{\costDynPro_{y}}
}}
\left\{    
d_x(\boldsymbol{\balanceDynPro_{x}}, \boldsymbol{\costDynPro_x})  + d_y(\boldsymbol{\balanceDynPro_{y}}, \boldsymbol{\costDynPro_y})  
\right\}.
\end{align*}
\end{proof}
Finally, a robust flow in SP digraph $G$ is obtained by backtracking the steps of the DP, and considering the demand label associated to the SP tree's root $r$. 
\begin{lemma}\label{LemmaDPLabelWurzzelknoten}
Let $\boldsymbol{f}$ be an optimal robust $\boldsymbol{b}$-flow in SP digraph $G_r$. For the cost it holds that
\begin{align}
c(\boldsymbol{f})
= \min
	\Big \{\hat{c} \ \big|\big. \ \exists \ \boldsymbol{\costDynPro_r}\in \{0,\ldots,C_r\}^{|\Lambda|}: \ \max_{\lambda\in \Lambda}\costDynPro^\lambda_{r}=\hat{c} \ \wedge \ d_r(\boldsymbol{b}(o_r), \boldsymbol{\costDynPro_r})=0\Big \}
	\label{DPLoesung}
\end{align}
with $C_r=\sum_{a\in A(G_r)} c(a)\cdot u(a)$.
\end{lemma}
\begin{proof}
For all vertices $v\in V(G)\setminus\{\sinkSPGraph\}$, the flow balance constraints of the \ProblemName{} problem are ensured by constraints~\eqref{DPConstr2} of the \RestrictedProblemName{r} problem with $\boldsymbol{\balanceDynPro_{r}}= \boldsymbol{b}(o_r)$. 
The consistent flow and capacity constraints as well as the integer conditions of the \ProblemName{} problem are one to one included in the \RestrictedProblemName{r} problem by constraints~\eqref{DPConstr3},~\eqref{DPConstr4} and~\eqref{DPConstr5}, respectively. 
Accordingly, every feasible solution to the \RestrictedProblemName{r} problem is also a feasible solution to the \ProblemName{} problem. 

However, the \RestrictedProblemName{r} problem contains one additional set of constraints, namely constraints~\eqref{DPConstr1}. 
Constraints~\eqref{DPConstr1} control whether the cost of a flow is equal to the budget.  
For this reason, we look for a budget $\boldsymbol{\costDynPro_r}\in \{0,\ldots,C_r\}^{|\Lambda|}$ for which a feasible solution to the \RestrictedProblemName{r} problem exists, i.e., for which $d_r(\boldsymbol{b}(o_r), \boldsymbol{\costDynPro_r})=0$ holds true.
This solution corresponds to a robust $\boldsymbol{b}$-flow $\boldsymbol{f}$ with cost $c(\boldsymbol{f})=\max_{\lambda\in \Lambda} c(f^\lambda)= \max_{\lambda \in \Lambda}\costDynPro^\lambda_{r}$.
Therefore, we are interested in the minimum maximum budget needed among all scenarios $\lambda\in \Lambda$ which we obtain by expression \eqref{DPLoesung}.
\end{proof}
After all, we analyze the runtime of the DP.  
\begin{theorem}\label{Theorem:PseudopolynomialDP}
Let $(G,u,c,\boldsymbol{b})$ be a \ProblemName{} instance where $G$ is an SP digraph with origin $o$.
Using the DP described, the \ProblemName{} problem can be solved in $\mathcal{O}(|A(G)|(U+1)^{2|\Lambda|}( C+1)^{2|\Lambda|})$ time where \- $U:=\sum_{ a=(\sourceSPGraph,v)\in A(G)} u(a)$ and $C:=\sum_{ a\in A(G)} c(a)\cdot u(a)$ holds.
\end{theorem} 
\begin{proof}
The correctness of the algorithm follows from Lemmas \ref{LemmaDPInitalisierung}-\ref{LemmaDPLabelWurzzelknoten}. 
Considering the runtime, first of all, we mention that the representation of an SP digraph $G$ by its SP tree $T$ can be computed in $\mathcal{O}(|A(G)|)$ \citep{Valdes:1979:RSP:800135.804393}.
At every SP tree's vertex $v\in V(T)$ demand labels for all supplies $\boldsymbol{\balanceDynPro_v}$ and budgets $\boldsymbol{\costDynPro_v}$ need to be calculated where the number of combinations is limited by $(U+1)^{|\Lambda|} \cdot (C+1)^{|\Lambda|}$. 
As SP tree $T$ of SP digraph $G$ has exactly $|V(T)|=2|A(G)|-1$ vertices, we have to compute $\mathcal{O}(2|A(G)|\cdot (U+1)^{|\Lambda|} \cdot (C+1)^{|\Lambda|})$ demand labels.
It remains to bound the complexity for computing the demand labels. 
If $v\in V(T)$ is an $L$-vertex, computing the corresponding demand labels is clearly in $\mathcal{O}(1)$. 
If $v\in V(T)$ is an $S$- or $P$-vertex, we need to compute the minimum of $(U+1)^{|\Lambda|}(C+1)^{|\Lambda|}$ sums which is in $\mathcal{O}((U+1)^{|\Lambda|}(C+1)^{|\Lambda|})$.
In total, we obtain a runtime of $\mathcal{O}(2|A(G)|\cdot (U+1)^{2|\Lambda|} \cdot (C+1)^{2|\Lambda|})$.
\end{proof}
By reason of Theorem~\ref{Theorem:PseudopolynomialDP}, the pseudo-polynomial runtime of the DP follows. 
Together with the result of Theorem~\ref{Theorem:SPComplexity} we obtain the following corollary. 
\begin{corollary}
The decision version of the \ProblemName{} problem on networks based on SP digraphs with multiple sources and multiple sinks is weakly $\mathcal{NP}$-complete and can be solved by the presented DP in pseudo-polynomial time.
\end{corollary}

\subsection{Special Case of Unique Source and Unique Sink Networks}
\label{Subsec:SpecialCaseUniqueSourceUniqueSinkCapacities}
In this section, we provide a polynomial time algorithm for the special case of networks based on SP digraphs with a unique source and a unique sink.  
The core idea of the algorithm is based on the algorithm of Bein et al.~(\citeyear{bein1985minimum}) which iteratively sends flow along shortest paths to solve the MCF problem. 
Before we propose a generalized algorithm for the \ProblemName{} problem, we investigate properties of an optimal robust flow in the networks considered.  
In particular, we study the cost and show that we can restrict the statement of Lemma \ref{lem:UniqueSourceUniqueSinkLimitationCost}. 

We start with introducing the notations and definitions needed.
Let us consider a \ProblemName{} instance $(G,u,c,\boldsymbol{b})$ where $G$ is an SP digraph with origin $\sourceSPGraph$ and target $\sinkSPGraph$.
As SP digraphs are acyclic, we assume without loss of generality that the unique source complies with origin $\sourceSPGraph$ and the unique sink complies with target $\sinkSPGraph$. 
Due to Lemma \ref{lem:UniqueSourceUniqueSinkLimitationCost}, we limit our efforts to a set of two scenarios, i.e., $\Lambda=\{1,2\}$.
For convenience, we introduce a demand vector $\boldsymbol{d}=(d^1,d^2)$, consisting of the number of flow units that, according to the balances $\boldsymbol{b}$, are supplied from the unique source and demanded by the unique sink, i.e., $d^1:=b^1(\sourceSPGraph)=-b^1(\sinkSPGraph)$ and $d^2:=b^2(\sourceSPGraph)=-b^2(\sinkSPGraph)$.
Without loss of generality, let $d^1$ and $d^2$ be given such that $d^1\leq d^2$ holds true.
Further, let $H\subseteq G$ be an SP subgraph with origin $o_H$. 
We denote the flow value of a given flow $f^\lambda$, $\lambda\in \Lambda$ entering subgraph $H$ by $\Demand{}{f}{\lambda}{|H}$ such that the following holds
$$\Demand{}{f}{\lambda}{|H}:= \sum_{a=(o_H,w)\in A(H)} f^\lambda(a).$$ 

Using these notations and definitions we aim at investigating the cost of an optimal robust flow.
In contrast to networks based on acyclic digraphs with a unique source and a unique sink, see Example~\ref{expl:ExampleCost}, for the special case considered in this section it is sufficient to concentrate on the cost of the last scenario flow.
Before we prove this statement, we need the following two auxiliary lemmas.
\begin{restatable}[]{lemma}{RedirectionFlow}
	\label{lem:RedirectionFlow}
	Let $G$ be an SP digraph which is composed by subgraphs $G_1$ and $G_2$, and let $(G,u,c,\boldsymbol{b})$ be a corresponding \ProblemName{} instance.
	There exists an optimal robust $\boldsymbol{b}$-flow $\boldsymbol{f}=(f^1,f^2)$ for which $\Demand{2}{f}{2}{|G_1}\geq \Demand{1}{f}{1}{|G_1}$ and $\Demand{}{f}{2}{|G_2}\geq \Demand{}{f}{1}{|G_2}$ hold true. 
\end{restatable}
By reason of the consistent flow constraints, the statement is not apparent. 
Due to the length, the proof is moved to Appendix~\ref{Appendix}.

\begin{lemma}\label{lem:FlowRelationScenarios}
Let $(G,u,c,\boldsymbol{b})$ be a \ProblemName{} instance where $G$ is an SP digraph.
There exists an optimal robust $\boldsymbol{b}$-flow $\boldsymbol{f}=(f^1,f^{2})$ such that $f^2(a)\geq f^1(a)$ holds true for all arcs $a\in A(G)$. 
\end{lemma}
\begin{proof}
Let $T$ be the SP tree of SP digraph $G$.
As we consider a digraph with a unique source and a unique sink, the statement of Lemma~\ref{lem:RedirectionFlow} can be recursively transferred to subgraphs $G_v\subseteq G$ associated to the SP tree's vertices $v\in V(T)$, i.e., $\Demand{2}{f}{2}{|G_v}\geq \Demand{1}{f}{1}{|G_v}$ for all $v\in V(T)$.
Consequently, $f^2(a)\geq f^1(a)$ holds true for all arcs $a\in A(G)$.  
\end{proof}
Note that, in general, the statement of Lemma \ref{lem:FlowRelationScenarios} is not true for acyclic digraphs as Example~\ref{expl:ExampleCost} shows. We are now able to prove the following crucial lemma regarding the cost of a robust flow.
\begin{lemma}\label{lem:CostSPGraphs}
Let $(G,u,c,\boldsymbol{b})$ be a \ProblemName{} instance where $G$ is an SP digraph.
There exists an optimal robust $\boldsymbol{b}$-flow $\boldsymbol{f}=(f^1,f^{2})$ whose cost is determined by the cost of the last scenario flow, i.e., 
$c(\boldsymbol{f})=\max\{c(f^1),c(f^{2})\}=c(f^{2})$.
\end{lemma}
\begin{proof}
By Lemma \ref{lem:FlowRelationScenarios}, there exists an optimal robust $\boldsymbol{b}$-flow $\boldsymbol{f}=(f^1, f^2)$ such that $f^2(a)\geq f^1(a)$ holds for all arcs $a\in A(G)$. 
The scenario flows cause the following cost 
\begin{align*}
c(f^1)= \sum_{a\in A} c(a)\cdot f^1(a) \leq \sum_{a\in A} c(a)\cdot f^2(a) = c(f^2),
\end{align*}
from which the statement immediately follows.  
\end{proof}
By reason of Lemma~\ref{lem:CostSPGraphs}, we concentrate on the last scenario in the following.
Firstly, we note that a last scenario flow needs to send demand $d^2-d^1$ in subgraph $G-\fix$, and we refer to this demand as \textit{excess demand}. 
Before we present a further useful property of the last scenario flow regarding its excess demand and a shortest path in subgraph $G-\fix$, we need the following auxiliary lemma. 

\begin{restatable}[]{lemma}{BellmanOptikriteriumSP}
	\label{lem:BellmanOptikriteriumSP}
	Let $G$ be a series composition of SP digraphs $G_1$ and $G_2$, and let $\mathcal{I}$ be a corresponding \ProblemName{} instance.
	Then, let $\mathcal{I}_1$ and $\mathcal{I}_2$ be the \ProblemName{} instances which are obtained by restricting instance $\mathcal{I}$ to subgraphs $G_1$ and $G_2$, respectively. 
	A solution $\boldsymbol{f}$ to instance $\mathcal{I}$ is optimal if and only if the solutions $\boldsymbol{f}_{|G_1}$ and $\boldsymbol{f}_{|G_2}$, which can be obtained by restricting $\boldsymbol{f}$ to subgraphs $G_1$ and $G_2$, are optimal to instances $\mathcal{I}_1$ and $\mathcal{I}_2$, respectively. 
\end{restatable}
A proof can be found in Appendix~\ref{Appendix}. 
Example~\ref{exmpl:SPBellmannsPrinciple} in Appendix~\ref{Appendix} shows that the SP property of the digraph is necessary for the truthfulness of Lemma~\ref{lem:BellmanOptikriteriumSP}.
Using Lemma~\ref{lem:BellmanOptikriteriumSP}, we formulate a useful property for an existing optimal robust flow. 
\begin{lemma}\label{ClaimPropertySecondFlow}
Let $G=(V,A=\fix\cup \free)$ be an SP digraph with origin $o$ and target $q$. Further, let $\mathcal{I}=(G,u,c,\boldsymbol{b})$ be a corresponding \ProblemName{} instance with demand $\boldsymbol{d}=(d^1,d^2)$, $d^2\geq d^1$. With respect to cost $c$, let $p$ be a shortest $(\sourceSPGraph,\sinkSPGraph)$-path in subgraph $G-\fix$ with its bottleneck value $u_p=\min_{a\in A(p)}u(a)$. There exists an optimal robust $\boldsymbol{b}$-flow $\boldsymbol{f}=(f^1,f^2)$ for which the following holds true
\begin{align}
		f^2(a)\geq \min \{u_p,d^2-d^1\} & \text{ for all }\ a\in A(p).  \label{property}
\end{align}
\end{lemma}
\begin{proof}
We prove the correctness of the statement by induction on the number of the digraph's arcs $m:=|A|$. 
For the beginning, if we consider a digraph consisting of one arc only, the statement is readily apparent. 
In the next step, we prove the statement for a digraph with $m+1$ arcs, providing that the statement holds true for all digraphs consisting of at most $m$ arcs. 
For this purpose, we distinguish between two cases.
\begin{figure}
	\begin{adjustbox}{max width=1\textwidth, max height=1\textheight}
		\subfloat[]{\label{fig:InductionSeriesCompositionMinFlowValueSP(a)}	\includegraphics{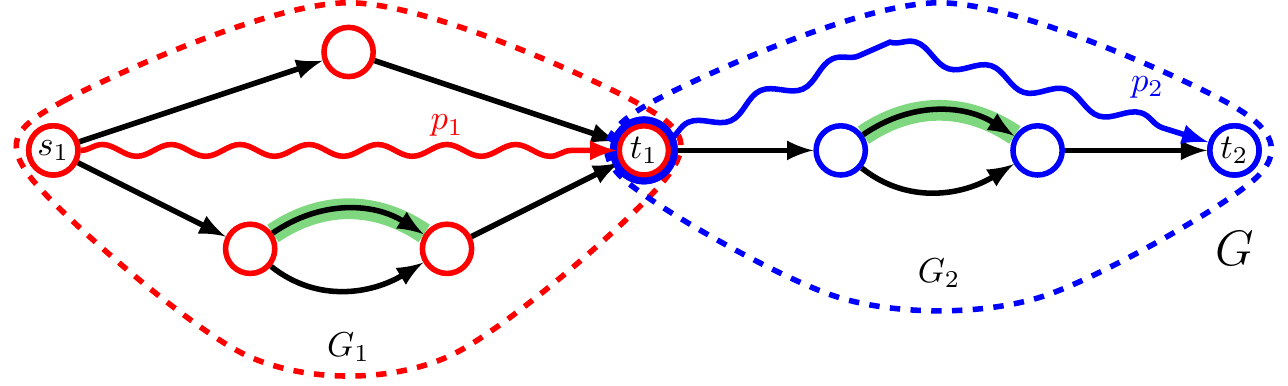}}
		\subfloat[]{\label{fig:InductionSeriesCompositionMinFlowValueSP(b)}	\includegraphics{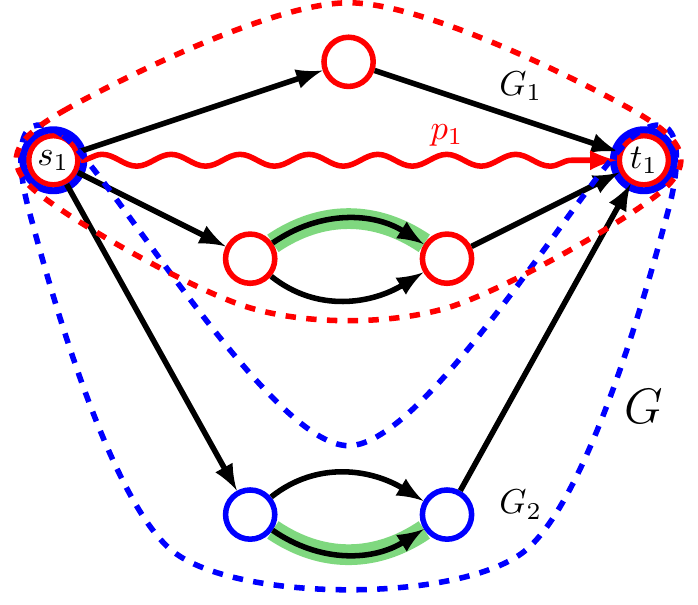}}
	\end{adjustbox}
	\centering\includegraphics{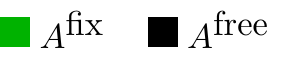}
	\caption{Shortest path $p$ in subgraph $G-\fix$ where digraph $G$ is a series (a) or parallel (b) composition}
\end{figure}

Firstly, we assume that $G$ is a series composition of SP digraphs $G_1$ and $G_2$. 
Therefore, the origin of digraph $G_1$ and the target of digraph $G_2$ are contracted to one vertex that we denote by $w$. 
Due to the composition of digraph $G$, a shortest $(\sourceSPGraph,\sinkSPGraph)$-path $p$ in subgraph $G-\fix$ is composed of a shortest $(\sourceSPGraph,w)$-path $p_{1}$ in subgraph $G_1-\fix$, and a shortest $(w,\sinkSPGraph)$-path $p_2$ in subgraph $G_2-\fix$, see Figure \ref{fig:InductionSeriesCompositionMinFlowValueSP(a)}. 
Considering subgraphs $G_1$ and $G_2$ separately, we obtain the \ProblemName{} instances $\mathcal{I}_1$ and $\mathcal{I}_2$, respectively.
By induction hypothesis there exist an optimal robust flow $\boldsymbol{f}_{|G_1}=(f^1_{|G_1},f^2_{|G_1})$ in subgraph $G_1$ and an optimal robust flow $\boldsymbol{f}_{|G_2}=(f^1_{|G_2},f^2_{|G_2})$ in subgraph $G_2$ satisfying
\begin{align*}
		f^2_{|G_1}(a)\geq \min \{u_{p_1},d^2-d^1\} & \mbox{ for all arcs $a\in A(p_{1})$}, \\
		f^2_{|G_2}(a)\geq \min \{u_{p_2},d^2-d^1\} & \mbox{ for all arcs $a\in A(p_2)$}.
\end{align*}
By Lemma \ref{lem:BellmanOptikriteriumSP}, the composed flow $\boldsymbol{f}=(f^1,f^2)$ with $f^1:=f^1_{|G_1}+f^1_{|G_2}$ and $f^2:=f^2_{|G_1}+f^2_{|G_2}$ is an optimal robust $\boldsymbol{b}$-flow in digraph $G$ where the desired property is still satisfied. 

Secondly, we assume that $G$ is a parallel composition of SP digraphs $G_1$ and $G_2$.
Without loss of generality, let the shortest $(\sourceSPGraph,\sinkSPGraph)$-path $p$ be contained in subgraph $G_1-\fix$, see Figure \ref{fig:InductionSeriesCompositionMinFlowValueSP(b)}. 
Further, let $\boldsymbol{f}=(f^1,f^2)$ be an optimal robust $\boldsymbol{b}$-flow which sends demand $\boldsymbol{d}=(d^1,d^2)$ through digraph $G$ that satisfies without loss of generality the property of Lemma~\ref{lem:RedirectionFlow}, i.e., $\Demand{}{f}{2}{|G_i} - \Demand{}{f}{1}{|G_i}\geq 0$ for $G_i$, $i\in \{1,2\}$.
If the optimal flow $\boldsymbol{f}$ does not satisfy the property~\eqref{property}, applying the following procedure leads to the desired result.
We consider the subgraphs $G_1$ and $G_2$ separately, resulting in \ProblemName{} instances $\mathcal{I}_1$ and $\mathcal{I}_2$, respectively.
To define how much demand $\boldsymbol{d}_{|G_i}$ is supposed to be sent through subgraph $G_i$ of instance $\mathcal{I}_i$, we exploit the partition of demand $\boldsymbol{d}$ of the optimal flow $\boldsymbol{f}$, i.e., $\boldsymbol{d}_{|G_i}=(d^1_{|G_i}, d^2_{|G_i}):=(\Demand{}{f}{1}{|G_i},\Demand{}{f}{2}{|G_i})$ for $i\in \{1,2\}$. 
Considering subgraph $G_1$, by induction hypothesis there exists an optimal robust flow $\boldsymbol{\tilde{f}}=(\tilde{f}^1,\tilde{f}^2)$ which sends demand $\boldsymbol{d}_{|G_1}=(d^1_{|G_1}, d^2_{|G_1})$ and satisfies
\begin{align*}
		\tilde{f}^2(a)	\geq \min \{ u_p, d^2_{|G_1}- d^1_{|G_1} \}  & \mbox{ for all arcs $a\in A(p)$}. 
\end{align*}
Further, let an optimal robust flow $\boldsymbol{\hat{f}}$ be given which sends demand $\boldsymbol{d}_{|G_2}$ through subgraph $G_2$. 
By composing flows $\boldsymbol{\tilde{f}}$ and $\boldsymbol{\hat{f}}$, we obtain a robust $\boldsymbol{b}$-flow $\boldsymbol{\overline{f}}=(\overline{f}^1,\overline{f}^2)$ with scenario flows $\overline{f}^1:=\tilde{f}^1+\hat{f}^1$ and $\overline{f}^2:=\tilde{f}^2+\hat{f}^2$ for instance $\mathcal{I}$. 
Flow $\boldsymbol{\overline{f}}$ is optimal as there exists an optimal robust flow with the same partition $\boldsymbol{d}_{|G_1}$ and $\boldsymbol{d}_{|G_2}$ of demand $\boldsymbol{d}$ between subgraphs $G_1$ and $G_2$, and as flows $\boldsymbol{\tilde{f}}$ and $\boldsymbol{\hat{f}}$ are optimal themselves.
It remains to prove that $\overline{f}^2(a)\geq \min \{ u_p ,  d^2-d^1  \}$ holds for all arcs $a\in A(p)$. 
We distinguish between the following two cases. 

Firstly, we consider the case where $d^2_{|G_1}- d^1_{|G_1}\geq \min \{ u_p ,  d^2-d^1  \}$ holds true. 
As $\overline{f}^2(a)=\tilde{f}^2(a)$ holds for all arcs $a\in A(G_1)$ by construction, the desired property results immediately for all arcs $a\in A(p)\subseteq A(G_1)$ as shown by the following
\begin{align*}
		\overline{f}^2(a)=\tilde{f}^2(a)	
						&\geq		\min \{ u_p, d_{|G_1}^2- d_{|G_1}^1 \}\\
						&\geq  \min \{ u_p,\min \{ u_p, d^2-d^1\} \} \\
						&=  \min \{ u_p, d^2-d^1 \}.
\end{align*}
Secondly, we consider the case where $d^2_{|G_1}- d^1_{|G_1}< \min \{ u_p ,  d^2-d^1  \}$ holds true. 
Assume $\overline{f}^2(a)< \min \{ u_p , d^2-d^1     \} $ is true for one arc $a\in A(p)\subseteq A(G_1)$. 
We redirect the last scenario flow $\overline{f}^2$ of robust flow $\boldsymbol{\overline{f}}$ such that demand of $\min \{ u_p , d^2-d^1     \}$ is sent along the shortest path $p$ in subgraph $G-\fix$. 
As demand $d^2-d^1$ needs to be sent in subgraph $G-\fix$ in any case, and $A(p)\subseteq \free$ holds, the resulting robust flow is still feasible, satisfies the desired property~\eqref{property}, and its cost is not increased. 
\end{proof}
Based on the derived knowledge by the presented lemmas,  we can finally present an algorithm that solves the \ProblemName{} problem on networks based on SP digraphs with a unique source and a unique sink. 
\begin{algorithm}[H]
\begin{lyxlist}{Method:}
\item [{Input:}] SP digraph $G=(V,A=A^{\text{fix}}\cup A^{\text{free}})$, instance $\mathcal{I}=(G,u,c,\boldsymbol{b})$, demand $\boldsymbol{d}$ 
\item [{Output:}] Robust minimum cost $\boldsymbol{b}$-flow $\boldsymbol{f}$
\item [{Method:}]~	
\begin{algorithmic}[1]
		\State{Compute a minimum cost flow $f^\prime$ that sends demand $d^2-d^1$ in subgraph $G-\fix$ with respect to capacity $u$ and cost $c$}
		\State{Let $u^\prime$ be the capacity which results from reducing the capacity $u$ of all arcs of digraph $G$ that are used by flow $f^\prime$.
		By means of the Greedy Algorithm of Bein et al.~(\citeyear{bein1985minimum}) compute a minimum cost flow $f^{\prime\prime}$ that sends demand $d^1$ in digraph $G$  with respect to capacity $u^\prime$ and cost $c$, i.e., flow is sent along shortest paths that still have positive bottleneck values}
		\State{Set $f^1:=f^{\prime\prime}$ and $f^2:=f^\prime+ f^{\prime\prime}$}\\
		\Return $\boldsymbol{b}$-flow $\boldsymbol{f}=(f^1,f^2)$
\end{algorithmic}
\end{lyxlist}
\caption{}\label{alg:SPUniqueSourceUniqueSinkCapacities}
\end{algorithm}
Basically, Algorithm~\ref{alg:SPUniqueSourceUniqueSinkCapacities} computes a flow by sending the excess demand in subgraph $G-\fix$ first, and subsequently, by sending the demand through digraph $G$ which is sent in both scenarios. 
Composing the computed flows to a robust flow leads to an optimal solution obtained in polynomial time as the following theorem shows.
\begin{theorem}
Let $G$ be an SP digraph, and let $\mathcal{I}=(G,u,c,\boldsymbol{b})$ be a corresponding \ProblemName{} instance with demand $\boldsymbol{d}=(d^1,d^2)$, $d^2\geq d^1$.
Algorithm~\ref{alg:SPUniqueSourceUniqueSinkCapacities} computes an optimal robust $\boldsymbol{b}$-flow for demand $\boldsymbol{d}$ in polynomial time. 
\end{theorem}
\begin{proof}
We prove the statement by induction on the excess demand, i.e., $k:=d^2-d^1\geq 0$. 
For the beginning, we consider the case where $k=0$ holds. 
As the excess demand is zero, the same amount of flow needs to be sent in both scenarios. 
Thus, sending the excess demand in subgraph $G-\fix$ in step $1$ is omitted.
In step $2$, a minimum cost flow that sends demand $d^1$ through digraph $G$ is computed by the Greedy Algorithm of Bein et al.~(\citeyear{bein1985minimum}).
A feasible robust flow results whose scenario flows are equal. 
The robust flow is optimal by the correctness of the Greedy Algorithm of Bein et al.

For the induction step, let $\boldsymbol{\tilde{f}}=(\tilde{f}^1,\tilde{f}^2)$ be an optimal robust $\boldsymbol{b}$-flow for instance $\mathcal{I}$ which satisfies without loss of generality the properties of Lemmas~\ref{lem:RedirectionFlow}~--~\ref{ClaimPropertySecondFlow}, i.e., in particular, $\tilde{f}^2(a)\geq \overline{u}:= \min \{u_p,k+1\}$ for $a\in A(p)$.
We consider the \ProblemName{} instance $\widehat{\mathcal{I}}:=(G,\hat{u},c,\boldsymbol{\hat{b}})$ with the adjusted capacity $\hat{u}$ and balances $\boldsymbol{\hat{b}}:=(b^1,\hat{b}^2)$. 
Capacity $\hat{u}$ is obtained by reducing the capacity $u$ of all arcs of path $p$ by $\overline{u}$, and accordingly updating the last scenario balances $b^2$ of the source and sink results in the new balances $\boldsymbol{\hat{b}}:=(b^1,\hat{b}^2)$. 
Further, we obtain the new demand $\boldsymbol{\hat{d}}=(d^1,\hat{d}^2)$ with $\hat{d}^2:=d^2-\overline{u} $. 
As the excess demand is less or equal to $k$ in instance $\widehat{\mathcal{I}}$, by induction hypothesis Algorithm~\ref{alg:SPUniqueSourceUniqueSinkCapacities} computes an optimal robust $\boldsymbol{\hat{b}}$-flow $\boldsymbol{\hat{f}}=(\hat{f}^1, \hat{f}^2)$ that sends demand $\boldsymbol{\hat{d}}=(d^1, \hat{d}^2)$.  
We note that robust flow $\boldsymbol{\hat{f}}$ also satisfies the properties of Lemmas~\ref{lem:RedirectionFlow}~--~\ref{ClaimPropertySecondFlow}.
In summary, we obtain that $\hat{f}^2$ is a flow sending demand $\hat{d}^2=d^2-\overline{u}$ for instance $\widehat{\mathcal{I}}$, and by assumption, $\tilde{f}^2$ is an optimal last scenario flow sending demand $d^2$ for instance $\mathcal{I}$. 
Furthermore, by assumption flow $\tilde{f}^2$ sends $\overline{u}$ demand along the shortest path $p$ in subgraph $G-\fix$. 
Overall, we obtain for the cost of flow $\hat{f}^2$ the following upper bound
\begin{align*}
c(\hat{f}^2)\leq c(\tilde{f}^2) - \overline{u} \cdot c(p). 
\end{align*}
By reformulating, we obtain
\begin{align*}
c(\hat{f}^2)+ \overline{u} \cdot c(p)\leq c(\tilde{f}^2), 
\end{align*} 
and together with the condition that the cost of both flows are determined by the last scenario flows, the following holds true
\begin{align*}
c(\boldsymbol{\hat{f}}) + \overline{u} \cdot c(p)\leq c(\boldsymbol{\tilde{f}}). 
\end{align*}
Consequently, flow $\boldsymbol{f}=(f^1,f^2)$ with scenario flows $f^1:=\hat{f}^1$ and $f^2:=\hat{f}^2 + \overline{f}$ is an optimal robust $\boldsymbol{b}$-flow sending demand $\boldsymbol{d}$ where flow $\overline{f}$ is defined by $\overline{f}(a):=\overline{u}$ for all arcs $a\in A(p)$. 
Moreover, flow $\boldsymbol{f}=(f^1,f^2)$ complies with the flow computed by the Algorithm~\ref{alg:SPUniqueSourceUniqueSinkCapacities} for instance $\mathcal{I}$. 

Finally, considering the algorithm's runtime, we compute a minimum cost flow that sends demand $d^2-d^1$ by the Minimum Mean Cycle-Cancel Algorithm in $\mathcal{O}(|A|^3|V|^2\log|V|)$ time~\citep{korte2012combinatorial}. Subsequently, we compute a flow that sends demand $d^1$ by the Greedy Algorithm of Bein et al.~(\citeyear{bein1985minimum}) in $\mathcal{O}(|A|\cdot|V|+|A|\log |A|)$ time. In total, computing a robust minimum cost $\boldsymbol{b}$-flow takes $\mathcal{O}(|A|^3|V|^2\log|V|+|A|\cdot|V|+|A|\log |A|)$ time. 
\end{proof}
Note that we cannot use the Greedy Algorithm of Bein et al.~(\citeyear{bein1985minimum}) in the first step of Algorithm~\ref{alg:SPUniqueSourceUniqueSinkCapacities} as $G-\fix$ might not be an SP digraph.
\section{Conclusion}
\label{Sec:Conclusion}
In this paper, we introduced the \ProblemName{} problem which is an extension of the MCF problem considering equal flow requirements and demand uncertainty.
We presented structural results which differentiate from well known results of the MCF problem. 
In particular, we showed that Dantzig and Fulkerson's Integral Flow Theorem~\citep{korte2012combinatorial} does not hold anymore. 
Furthermore, we proved that finding a feasible solution to the \ProblemName{} problem is strongly $\mathcal{NP}$-complete on acyclic digraphs even if a network with a unique source and a unique sink is considered for two scenarios only. 
However, we proved that the decision version of the \ProblemName{} problem is only weakly $\mathcal{NP}$-complete on SP digraphs, and proposed a pseudo-polynomial DP. 
For the special case of networks based on SP digraphs with a unique source and a unique sink, we provided an algorithm running in polynomial time. 

For future work, we will study the \ProblemName{} problem for further graph classes as digraphs with bounded treewidth. 

\bibliography{Quellen.bib}   
\renewcommand{\appendixname}{}
\renewcommand*{\thesection}{\appendixname~\Alph{section}}
\appendix
\section{Appendix}
\label{Appendix}
\RedirectionFlow*
\begin{proof}
	Let $\boldsymbol{f}=(f^1,f^2)$ be an optimal robust $\boldsymbol{b}$-flow which sends demand $\boldsymbol{d}=(d^1,d^2)$ through digraph $G$.
	We distinguish whether digraph $G$ is a series or parallel composition of subgraphs $G_1$ and $G_2$. 
	For the case that digraph $G$ is serially composed the validity of the statement is apparent as we consider a network with a unique source and a unique sink, and $d^2\geq d^1$ holds. 
	In case that digraph $G$ is parallelly composed the following is true.
	If $d^2=d^1$ holds, the statement is again apparent. 
	Otherwise, if $d^2> d^1$ holds, $\Demand{}{f}{2}{|G_i}> \Demand{}{f}{1}{|G_i}$ also holds true for at least one of the subgraphs $G_i$, $i\in \{1,2\}$.
	Without loss of generality, let $G_1$ be the subgraph for which $\Demand{}{f}{2}{|G_1}> \Demand{}{f}{1}{|G_1}$ holds true, and in return, assume that $\Demand{}{f}{2}{|G_2}< \Demand{}{f}{1}{|G_2}$ holds. 
	In the following, we provide a procedure by which we redirect a proportion of the scenario flows $f^1$ or $f^2$ such that the desired property holds. 
	
	In the first step, we define two new scenario flows $\tilde{f}$ and $\hat{f}$ that send demand $\Demand{}{\tilde{f}}{}{|G}=d^1$ and $\Demand{}{\hat{f}}{}{|G}=d^2$ through digraph $G$, respectively.
	Flow $\tilde{f}$ corresponds to a first scenario flow which is obtained by redirecting a proportion of flow $f^1$ from subgraph $G_2$ to subgraph $G_1$ such that $\Demand{}{f}{2}{|G_2}\geq \Demand{}{\tilde{f}}{}{|G_2}$ holds true, see Figure~\ref{fig:ExmplRedirectingParallel2}. 
	\begin{figure}
		\begin{adjustbox}{max width=1\textwidth, max height=1\textheight}
			\centering	\includegraphics{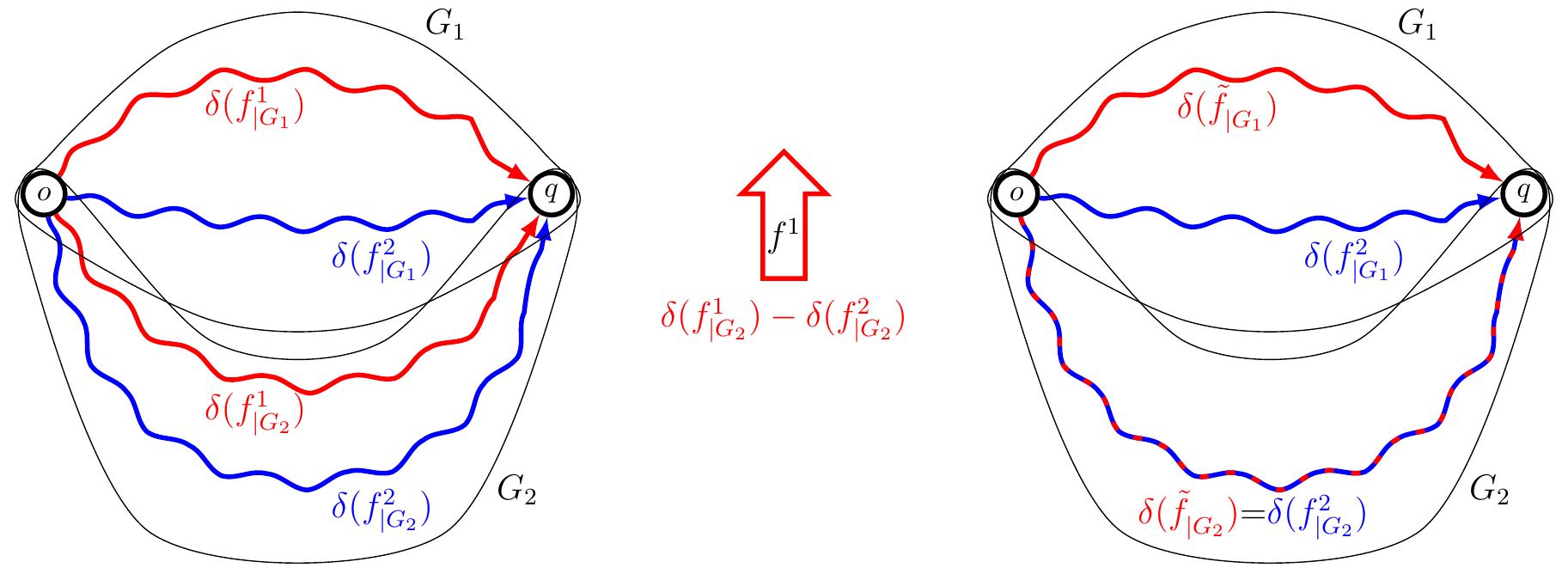}
		\end{adjustbox}
		\caption{Shifting $\Demand{}{f}{1}{|G_2}- \Demand{}{f}{2}{|G_2}$ demand of the first scenario flow $f^1$ from $G_2$ to $G_1$ leads to a new flow $\tilde{f}$ that satisfies the desired property}
		\label{fig:ExmplRedirectingParallel2}
	\end{figure}
	More precisely, flow $\tilde{f}:= \tilde{f}_{|G_1} + \tilde{f}_{|G_2}$ is defined in such a way that demand $\tilde{d}_1:=\Demand{}{\tilde{f}}{}{|G_1}=\Demand{}{f}{1}{|G_1} + (\Demand{}{f}{1}{|G_2}- \Demand{}{f}{2}{|G_2})$ is sent through subgraph $G_1$, and demand $\tilde{d}_2:=\Demand{}{\tilde{f}}{}{|G_2}=\Demand{}{f}{2}{|G_2}$ through subgraph $G_2$. 
	Considering subgraph $G_1$, by assumption and definition it holds $\Demand{}{f}{1}{|G_1}< \tilde{d}_1 < \Demand{}{f}{2}{|G_1} $.
	Following Lemma \ref{lem:UniqueSourceUniqueSinkLimitationCost}, we compute a robust flow $\boldsymbol{\overline{f}}=(f^1_{|G_1},\tilde{f}_{|G_1},f^2_{|G_1})$ sending demand $\boldsymbol{\overline{d}}=(\Demand{}{f}{1}{|G_1},\tilde{d}_1, \Demand{}{f}{2}{|G_1}) $ through subgraph $G_1$ and causing cost of $c(\boldsymbol{\overline{f}})=\max\{c(f^1_{|G_1}), c(f^2_{|G_1})\}$. 
	We further set $\tilde{f}_{|G_2}:= f^2_{|G_2}$ such that the overall cost of flow $\tilde{f}$ is estimated as follows
	\begin{align*}
	c(\tilde{f}_{|G_1})&\leq\max\{c(f^1_{|G_1}), c(f^2_{|G_1})\},\\
	c(\tilde{f}_{|G_2})&= c(f^2_{|G_2}).
	\end{align*}
	
	Flow $\hat{f}$ in turn corresponds to a last scenario flow which is obtained by redirecting a proportion of flow $f^2$ from subgraph $G_1$ to subgraph $G_2$ such that $\Demand{}{\hat{f}}{}{|G_2}\geq \Demand{}{f}{1}{|G_2}$ holds true, see Figure~\ref{fig:ExmplRedirectingParallel1}.
	\begin{figure}
		\begin{adjustbox}{max width=1\textwidth, max height=1\textheight}
			\includegraphics{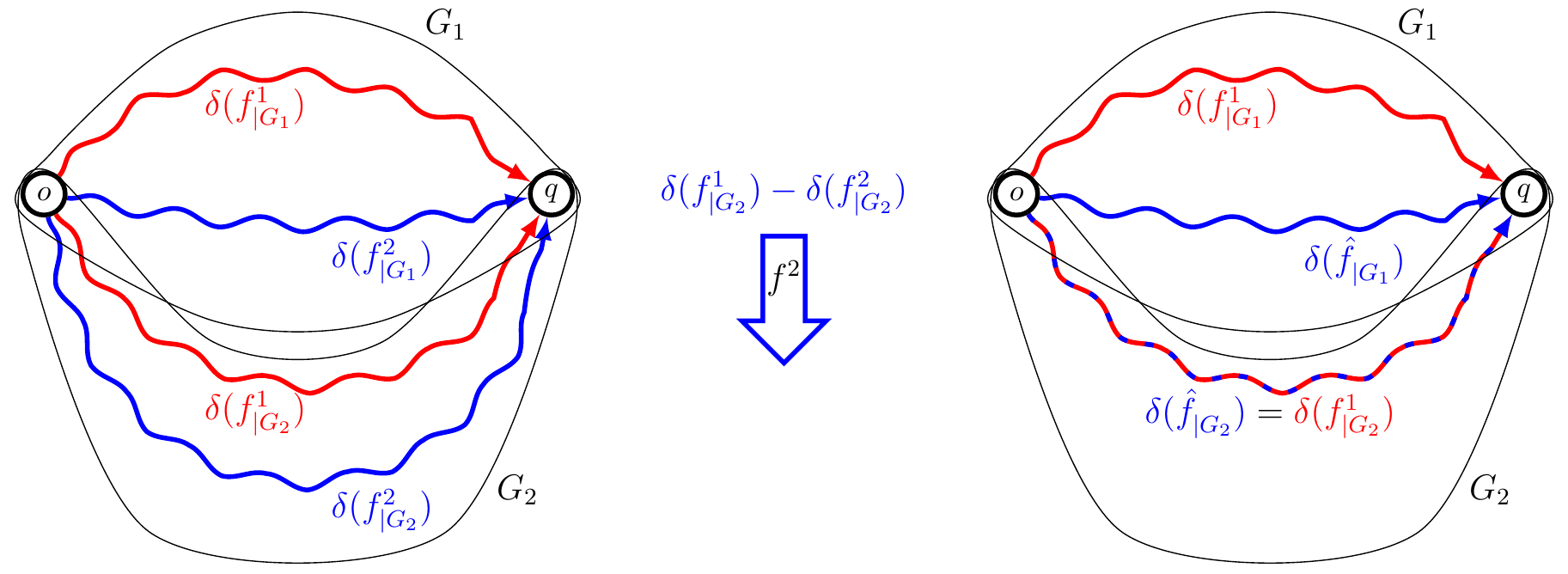}
		\end{adjustbox}
		\caption{Shifting $\Demand{}{f}{1}{|G_2}- \Demand{}{f}{2}{|G_2}$ demand of the last scenario flow $f^2$ from $G_1$ to $G_2$ leads to a new flow $\hat{f}$ that satisfies the desired property}
		\label{fig:ExmplRedirectingParallel1}
	\end{figure}
	More precisely, we define the scenario flow $\hat{f}:= \hat{f}_{|G_1} + \hat{f}_{|G_2}$ such that demand $\hat{d}_1:=\Demand{}{\hat{f}}{}{|G_1}=\Demand{}{f}{2}{|G_1} - (\Demand{}{f}{1}{|G_2}- \Demand{}{f}{2}{|G_2})$ is sent through subgraph $G_1$, and demand $\hat{d}_2:=\Demand{}{\hat{f}}{}{|G_2} =\Demand{}{f}{1}{|G_2}$ through subgraph $G_2$. 
	Considering subgraph $G_1$, by assumption and definition it holds $\Demand{}{f}{1}{|G_1}< \hat{d}_1 < \Demand{}{f}{2}{|G_1} $. 
	Following Lemma \ref{lem:UniqueSourceUniqueSinkLimitationCost}, we compute a robust flow $\boldsymbol{\underline{f}}=(f^1_{|G_1},\hat{f}_{|G_1},f^2_{|G_1})$ sending demand $\boldsymbol{\underline{d}}=(\Demand{}{f}{1}{|G_1},\hat{d}_1, \Demand{}{f}{2}{|G_1}) $ in subgraph $G_1$ and causing cost of $c(\boldsymbol{\underline{f}})=\max\{c(f^1_{|G_1}), c(f^2_{|G_1})\}$. 
	We further set $\hat{f}_{|G_2}:= f^1_{|G_2}$ and obtain the following estimations of the cost   
	\begin{align*}
	c(\hat{f}_{|G_1})&\leq\max\{c(f^1_{|G_1}), c(f^2_{|G_1})\}, \\
	c(\hat{f}_{|G_2})&= c(f^1_{|G_2}).
	\end{align*}
	
	In the next step, we construct two new robust $\boldsymbol{b}$-flows $\boldsymbol{f_a}:=(f^1,\hat{f})$ and $\boldsymbol{f_b}:=(\tilde{f},f^2)$ which are obtained by redirecting the scenario flows of the optimal robust $\boldsymbol{b}$-flow $\boldsymbol{f}$. 
	The robust flows $\boldsymbol{f_a}$ and $\boldsymbol{f_b}$ are feasible by construction of flows $\hat{f}$ and $\tilde{f}$ and each sends demand $\boldsymbol{d}$ through digraph $G$.
	If we show that $\min \{c(\boldsymbol{f_a}), c(\boldsymbol{f_b}) \}\leq c(\boldsymbol{f})$ holds true, we can redirect the optimal robust flow $\boldsymbol{f}$ analogous to either robust flow $\boldsymbol{f_a}$ or $\boldsymbol{f_b}$ such that the desired property is satisfied but the cost is not changed. 
	We distinguish whether the first or last scenario flow of the optimal robust solution $\boldsymbol{f}$ is more expensive.
	Firstly, we assume that the first scenario flow $f^1$ is more expensive than the last scenario flow $f^2$, i.e., $c(\boldsymbol{f})=\max\{c(f^1), c(f^2)\}=c(f^1)$.
	Further, we distinguish between the following two cases. 
	\begin{itemize}
		\item[] \underline{1. Case:} $c(f^1_{|G_1})\geq c(f^2_{|G_1})$\newline 
		To prove the statement $\min \{c(\boldsymbol{f_a}), c(\boldsymbol{f_b}) \}\leq c(\boldsymbol{f})$, it is sufficient to prove the statement $c(\boldsymbol{f_a})\leq c(\boldsymbol{f})$.
		By equivalent transformation we obtain 
		\begin{align*}
		&&c(\boldsymbol{f_a})				&\leq c(\boldsymbol{f})&\\
		\Leftrightarrow &&\max \{ c(f^1), c(\hat{f})\} 	&\leq \max \{ c(f^1), c(f^2)\}&\\
		\Leftrightarrow &&\max \{ c(f^1), c(\hat{f})\} 	&\leq  c(f^1).
		\end{align*}
		Consequently, we only need to prove that $c(\hat{f}) \leq c(f^1)$ holds. 
		Using the definition and cost estimation of flow $\hat{f}$, we can alternatively show the following
		\begin{align*}
		\max\{c(f^1_{|G_1}), c(f^2_{|G_1})\} + c(f^1_{|G_2})
		\leq c(f^1_{|G_1})+ c(f^1_{|G_2}).
		\end{align*}
		Equivalent transforming results in
		\begin{align*}
			&&	\max\{c(f^1_{|G_1}), c(f^2_{|G_1})\}  + c(f^1_{|G_2})&\leq c(f^1_{|G_1})+ c(f^1_{|G_2}) &\\
			\Leftrightarrow 				&&\max\{c(f^1_{|G_1}),  c(f^2_{|G_1})\} &\leq c(f^1_{|G_1}) &\\
			\Leftrightarrow 				&&c(f^1_{|G_1}) &\leq c(f^1_{|G_1}), 
		\end{align*}
		which is a true statement. 
		\item[] \underline{2. Case:} $c(f^1_{|G_1})< c(f^2_{|G_1})$\newline 
		To prove the statement $\min \{c(\boldsymbol{f_a}), c(\boldsymbol{f_b}) \}\leq c(\boldsymbol{f})$, it is sufficient to prove the statement $c(\boldsymbol{f_b})\leq c(\boldsymbol{f})$.
		For this case, the cost of the robust flow $\boldsymbol{f_b}=(\tilde{f}, f^2)$ is determined by flow $f^2$ as shown by the following
		\begin{align*}
		c(\tilde{f})
		=c(\tilde{f}_{|G_1})+c(\tilde{f}_{|G_2})
		&\leq \max\{c(f^1_{|G_1}), c(f^2_{|G_1})\}+c(f^2_{|G_2})\\
		&= c(f^2_{|G_1}) + c(f^2_{|G_2})
		=c(f^2).
		\end{align*}
		Accordingly, equivalent transformation results in
		\begin{align*}
		&&c(\boldsymbol{f_b})				&\leq c(\boldsymbol{f})&\\
		\Leftrightarrow &&\max \{ c(\tilde{f}), c(f^2)\} 	&\leq \max \{ c(f^1), c(f^2)\}&\\
		\Leftrightarrow && c(f^2) 	&\leq  c(f^1),
		\end{align*}
		which is a true statement for the present case. 
	\end{itemize}
	Secondly, we assume that the last scenario flow $f^2$ is more expensive than the first scenario flow $f^1$, i.e., $c(\boldsymbol{f})=\max\{c(f^1), c(f^2)\}=c(f^2)$.
	Further, we distinguish between the following two cases.
	\begin{itemize}
		\item[] \underline{1. Case:} $c(f^2_{|G_1})\geq c(f^1_{|G_1})$\newline 
		To prove the statement $\min \{c(\boldsymbol{f_a}), c(\boldsymbol{f_b}) \}\leq c(\boldsymbol{f})$, it is sufficient to prove the statement $c(\boldsymbol{f_b})\leq c(\boldsymbol{f})$.
		By equivalent transformation we obtain 
		\begin{align*}
		&&c(\boldsymbol{f_b})				&\leq c(\boldsymbol{f})&\\
		\Leftrightarrow &&\max \{ c(\tilde{f}), c(f^2)\} 	&\leq \max \{ c(f^1), c(f^2)\}&\\
		\Leftrightarrow &&\max \{ c(\tilde{f}), c(f^2)\} 	&\leq  c(f^2).
		\end{align*}
		Consequently, we only need to prove that $c(\tilde{f}) \leq c(f^2)$ holds. 
		Using the definition and cost estimation of flow $\tilde{f}$, we can alternatively show the following
		\begin{align*}
		\max\{c(f^1_{|G_1}), c(f^2_{|G_1})\} + c(f^2_{|G_2})
		\leq c(f^2_{|G_1})+ c(f^2_{|G_2}).
		\end{align*}
		Equivalent transforming results in
		\begin{align*}
					&& \max\{c(f^1_{|G_1}), c(f^2_{|G_1})\}  + c(f^2_{|G_2})&\leq c(f^2_{|G_1})+ c(f^2_{|G_2}) &\\
			\Leftrightarrow 			&&	\max\{c(f^1_{|G_1}),  c(f^2_{|G_1})\} &\leq c(f^2_{|G_1}) &\\
			\Leftrightarrow 		&&		c(f^2_{|G_1})&\leq   c(f^2_{|G_1}), 
		\end{align*}
		which is a true statement.

		\item[] \underline{2. Case:} $c(f^2_{|G_1})< c(f^1_{|G_1})$\newline 
		To prove the statement $\min \{c(\boldsymbol{f_a}), c(\boldsymbol{f_b}) \}\leq c(\boldsymbol{f})$, it is sufficient to prove the statement $c(\boldsymbol{f_a})\leq c(\boldsymbol{f})$.
		For this case, the cost of the robust flow $\boldsymbol{f_a}=(f^1, \hat{f})$ is determined by flow $f^1$ as shown by the following
		\begin{align*}
		c(\hat{f})
		=c(\hat{f}_{|G_1})+c(\hat{f}_{|G_2})
		\leq \max\{c(f^1_{|G_1}), c(f^2_{|G_1})\}+c(f^1_{|G_2})
		= c(f^1_{|G_1}) + c(f^1_{|G_2})
		=c(f^1).
		\end{align*}
		Accordingly, equivalent transformation results in
		\begin{align*}
		&&c(\boldsymbol{f_a})				&\leq c(\boldsymbol{f})&\\
		\Leftrightarrow && \max \{ c(f^1), c(\hat{f})\} 	&\leq \max \{ c(f^1), c(f^2)\}&\\
		\Leftrightarrow &&c(f^1) 	&\leq  c(f^2),
		\end{align*}
		which is a true statement for the present case. 
	\end{itemize}
	In summary, by redirecting the scenario flows of the optimal $\boldsymbol{b}$-flow $\boldsymbol{f}$ we obtain the desired property without changing the cost.
\end{proof}

\BellmanOptikriteriumSP*
\begin{proof}
	Without loss of generality, we assume that the robust flows given in this proof satisfy the property of Lemma~\ref{lem:CostSPGraphs}.
	Let $\boldsymbol{f}=(f^1,f^2)$ be an optimal robust flow for instance $\mathcal{I}$. 
	If we restrict flow $\boldsymbol{f}$ to subgraphs $G_1$ and $G_2$, feasible flows $\boldsymbol{f}_{|G_1}= (f^1_{|G_1}, f^2_{|G_1}) $ and $\boldsymbol{f}_{|G_2}= (f^1_{|G_2}, f^2_{|G_2}) $ result for instances $\mathcal{I}_1$ and $\mathcal{I}_2$, respectively. Furthermore, they still satisfy the property of Lemma~\ref{lem:CostSPGraphs}.
	Assume that flow $\boldsymbol{f}_{|G_1}$ is not optimal for instance $\mathcal{I}_1$. 
	Consequently, there exists an optimal robust flow $\boldsymbol{\tilde{f}}=(\tilde{f}^1, \tilde{f}^2)$ in subgraph $G_1$ with less cost, i.e.,
	$$c(\boldsymbol{\tilde{f}}) = \max \{c(\tilde{f}^1), c(\tilde{f}^2)\} =c(\tilde{f}^2) < c(f^2_{|G_1}) = \max \{ c(f^1_{|G_1})  ,  c(f^2_{|G_1}) \}  = c(\boldsymbol{f}_{|G_1}).$$
	However, this means that the composed flows $\hat{f}^1:=\tilde{f}^1+ f^1_{|G_2}$ and $\hat{f}^2:=\tilde{f}^2+ f^2_{|G_2}$ result in a feasible robust flow $\boldsymbol{\hat{f}}= (\hat{f}^1, \hat{f}^2)$ with cost 
	\begin{align*}
	c(\boldsymbol{\hat{f}})		&= \max \{ c(\hat{f}^1) , c(\hat{f}^2)  \} \\ 
	&= \max \{  c(\tilde{f}^1) + c(f^1_{|G_2}), c(\tilde{f}^2) + c(f^2_{|G_2})   \} \\
	&= c(\tilde{f}^2) + c(f^2_{|G_2})\\
	&< c(f^2_{|G_1}) + c(f^2_{|G_2}) =c(f^2)= c(\boldsymbol{f}),
	\end{align*} 
	which contradicts to the assumption. 
	The optimality of flow $\boldsymbol{f}_{|G_2}= (f^1_{|G_2}, f^2_{|G_2})$ follows for instance $\mathcal{I}_2$ due to the analog argumentation. 
	
	Conversely, let $\boldsymbol{f}_{|G_1}$ and $\boldsymbol{f}_{|G_2}$ be optimal flows for instances $\mathcal{I}_1$ and $\mathcal{I}_2$, respectively. 
	The composition of these flows results in a feasible robust flow $\boldsymbol{f}:=\boldsymbol{f}_{|G_1}+ \boldsymbol{f}_{|G_2}$ for instance $\mathcal{I}$ that causes cost of
	\begin{align*}
	c(\boldsymbol{f})	&= \max \{ c(f^1) , c(f^2)  \} \\
	&= \max \{ c(f^1_{|G_1}) + c(f^1_{|G_2} ),    c(f^2_{|G_1}) + c(f^2_{|G_2} )   \} \\
	&=  c(f^2_{|G_1}) + c(f^2_{|G_2} )\\
	&=  c(f^2).
	\end{align*}
	Assume the robust flow $\boldsymbol{f}$ is not optimal which in turn means that there exists an optimal robust flow $\boldsymbol{\tilde{f}}$ with less cost, i.e. $c(\boldsymbol{\tilde{f}})<c(\boldsymbol{f})$.
	As flows $\boldsymbol{f}_{|G_1}$ and $\boldsymbol{f}_{|G_2}$ are optimal for instances $\mathcal{I}_1$ and $\mathcal{I}_2$, respectively, $c(f^2_{|G_i}) \leq  c(\tilde{f}^2_{|G_i})$ holds true for both subgraphs $G_i$, $i\in \{1,2\}$. 
	Overall, we obtain 
	\begin{align*}
	c(\boldsymbol{f})	= c(f^2) = c(f^2_{|G_1}) + c(f^2_{|G_2} )  \leq  c(\tilde{f}^2_{|G_1}) + c(\tilde{f}^2_{|G_2} )  =  c(\tilde{f}^2)  = c(\boldsymbol{\tilde{f}}),
	\end{align*}
	which is a contradiction to the assumption.
\end{proof}

\begin{example}\label{exmpl:SPBellmannsPrinciple}
	\begin{figure}
			\begin{adjustbox}{max width=1\textwidth, max height=1\textheight}
			\centering\includegraphics{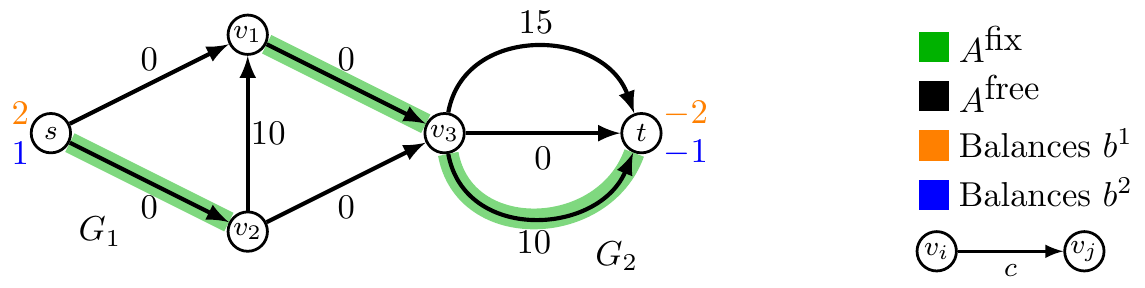}
		\end{adjustbox}
		\caption{An optimal solution in subgraph $G_1$ composed with an optimal solution in subgraph $G_2$ is not optimal in digraph $G$ composed by digraphs $G_1$ and $G_2$}
		\label{fig:BellOptiPrinzipSPGraphen}
	\end{figure}
	For a set of two scenarios $\Lambda=\{1,2\}$, let a network $(G,u,c,\boldsymbol{b})$ with capacity $u\equiv1$ be given where digraph $G$, its cost $c$, and the non-zero balances $\boldsymbol{b}$ are visualized in Figure~\ref{fig:BellOptiPrinzipSPGraphen}.
	An optimal solution $\boldsymbol{f}=(f^1, f^2)$ to the \ProblemName{} problem can be easily established.
	Considering the second scenario flow $f^2$ first, the only option to send two flow units from source $s$ to sink $t$ is along paths $sv_1v_3t$ and $sv_2v_3t$ due to the capacity constraints. 
	As the second scenario flow $f^2$ uses both fixed arcs in subgraph $G_1$, the first scenario flow $f^1$ must also send flow along these arcs. For this reason, the only option to send one flow unit from source $s$ to sink $t$ is along the path $sv_2v_1v_3t$.
	Concentrating on subgraph $G_1$, flow $f^1$ causes cost of ten while flow $f^2$ does not cause any cost. 
	Since the overall aim is to construct a robust $\boldsymbol{b}$-flow with minimum cost, flow $f^1$ sends the flow unit via the second parallel arc of multi-arc $(v_3,t)$ in subgraph $G_2$ causing zero cost. Flow $f^2$ also sends one flow unit via this arc, and additionally one flow unit via the first parallel arc of multi-arc $(v_3,t)$ causing cost of $15$. 
	In total, we obtain cost of
	$$
	c(\boldsymbol{f})
	= \max \{ c(f^1), c(f^2)\} 
	= \max \{ 10 + 0, 0+15\}
	=15
	=c(f^2).
	$$
	Due to the construction of digraph $G$, sending flow along paths from source $s$ to sink $t$ requires the usage of vertex $v_3$ which connects the subgraphs $G_1$ and $G_2$. 
	For this reason, we consider in the next step the \ProblemName{} problem on the subgraphs $G_1$ and $G_2$ separately. 
	Therefore, let $\mathcal{I}_1=(G_1,u,c,\boldsymbol{\tilde{b}})$ be the \ProblemName{} instance restricted to subgraph $G_1$ with newly defined balances by
	\begin{align*}
	\boldsymbol{\tilde{b}}(v)=
	\begin{cases}
	\boldsymbol{b}(v) \text{ for all }v\in V(G_1)\setminus\{v_3\},\\
	\boldsymbol{b}(t) \text{ for }v=v_3.
	\end{cases}
	\end{align*} 
	An optimal solution $\boldsymbol{\tilde{f}}=(\tilde{f}^1, \tilde{f}^2)$ to instance $\mathcal{I}_1$ is equal to solution $\boldsymbol{f}$ restricted to subgraph $G_1$, and causes cost of 
	$$
	c(\boldsymbol{\tilde{f}})
	= \max \{ c(\tilde{f}^1), c(\tilde{f}^2)\} 
	= \max \{ 10, 0\}
	=10.
	$$
	Further, let $\mathcal{I}_2=(G_2,u,c,\boldsymbol{\hat{b}})$ be the \ProblemName{} instance restricted to subgraph $G_2$ with balances $\boldsymbol{\hat{b}}(v_3)=\boldsymbol{b}(s)$ and $\boldsymbol{\hat{b}}(t)=\boldsymbol{b}(t)$. 
	An optimal solution $\boldsymbol{\hat{f}}=(\hat{f}^1, \hat{f}^2)$ to instance $\mathcal{I}_2$ is determined as follows.
	Both scenario flows $\hat{f}^1$ and $\hat{f}^2$ send one flow unit along the third parallel arc of multi-arc $(v_3,t)$ while the second scenario flow $\hat{f}^2$ additionally sends one flow unit along the second parallel arc. 
	This ends up in cost of
	$$
	c(\boldsymbol{\hat{f}})
	= \max \{ c(\hat{f}^1), c(\hat{f}^2)\} 
	= \max \{ 10, 10+0\}
	=10.
	$$
	Consequently, the optimal solution $\boldsymbol{\hat{f}}$ in subgraph $G_2$ causes less cost than the optimal solution $\boldsymbol{f}$ in digraph $G$ restricted to subgraph $G_2$ which causes cost of $15$. 
	
	Conversely, the solution, which results if optimal solutions $\boldsymbol{\tilde{f}}$ and $\boldsymbol{\hat{f}}$ to instances $\mathcal{I}_1$ and $\mathcal{I}_2$ are composed, is feasible but not optimal for instance $\mathcal{I}$. 
\end{example}

\end{document}